\documentclass[11pt]{amsart}
\usepackage[top=1in, bottom=1.25in, left=1.25in, right=1.25in]{geometry}
\usepackage{amssymb,amsmath,amsthm,mathrsfs,comment}
\usepackage{graphicx} % Required for inserting images
\usepackage{enumitem}
\usepackage{amsfonts}
\usepackage{color}
\usepackage{bbold}
\usepackage{hyperref}

\theoremstyle{plain}
\newtheorem{thm}{Theorem}[section]
\newtheorem*{thm*}{Theorem}
\newtheorem{defn}[thm]{Definition}
\newtheorem{prop}[thm]{Proposition}
\newtheorem{cor}[thm]{Corollary}
\newtheorem{lem}[thm]{Lemma}
\newtheorem{remark}[thm]{Remark}

\newcounter{constK}
\renewcommand{\theconstK}{{{\kappa}_{\arabic{constK}}}}
\newcommand{\constK}{\refstepcounter{constK}\theconstK}

\def\R{{\mathbb{R}}}
\def\N{{\mathbb{N}}}

\def\H{{\mathbb{H}}}

\def\Ad{{\mathrm{Ad}}}

\def\s{{\mathbf{s}}}

\def\c{{\mathbf{c}}}
\def\v{{\mathbf{v}}}

\def\vol{{\mathrm{vol}}}
\def\inj{{\mathrm{inj}}}
\def\SO{{\mathrm{SO}}}
\def\SL{{\mathrm{SL}}}

\def\Mat{{\mathrm{Mat}}}

\def\Lie{{\mathrm{Lie}}}
\def\dim{{\mathrm{dim}}}
\def\hdim{{\mathrm{Hdim}\text{ }}}

\newcommand {\norm}[1] {\left\| {#1} \right\|}
\def\be{\begin{equation}}
\def\ee{\end{equation}}
\newcommand {\absolute}[1] {\left| {#1} \right|}
\newcommand\eng{\mathcal E}
\DeclareMathOperator{\diff}{d} 
\DeclareMathOperator{\diam}{diam}

\newcounter{constE}
\renewcommand{\theconstE}{{{C}_{\arabic{constE}}}}
\newcommand{\constE}{\refstepcounter{constE}\theconstE}

\newcommand{\unm}{\rho}
\newcommand{\cpct}{\mathsf K}

\title{Quantitative finiteness of hyperplanes in hybrid manifolds}
\author{Ko W. Ohm}
\address{Department of Mathematics, University of California San Diego\\ 9500 Gilman Dr, La Jolla,
CA 92093, USA}
\email{kwohm@ucsd.edu}

\author{Anthony Sanchez}
\address{Department of Mathematics, University of California San Diego\\ 9500 Gilman Dr, La Jolla,
CA 92093, USA}
\email{ans032@ucsd.edu}

\subjclass[2020]{Primary 22F30; Secondary 37D40, 22E40\\
\emph{Key words and phrases: Homogeneous dynamics, Margulis function, hyperbolic manifolds.}}

\begin{document}

\maketitle
%\tableofcontents

\begin{abstract}
We prove a quantitative finiteness theorem for the number of totally geodesic hyperplanes of non-arithmetic hyperbolic $n$-manifolds that arise from a gluing construction of Gromov and Piatetski-Shapiro for $n\ge3$. This extends work of Lindenstrauss-Mohammadi in dimension 3. This follows from effective density theorem for periodic orbits of $\SO(n-1,1)$ acting on quotients of $\SO(n,1)$ by a  lattice for $n\ge3$. The effective density result uses a number of a ideas including Margulis functions, a restricted projection theorem, and an effective equidistribution result for measures on the horospherical subgroup that are nearly full dimensional.
\end{abstract}

\section{Introduction}

The geometry of non-arithmetic hyperbolic manifolds is mysterious in spite of how plentiful they are.  McMullen and Reid independently conjectured that such manifolds have only finitely many totally geodesic hyperplanes. In the recent advances of Bader--Fisher--Miller--Stover \cite{MR4250391} for dimensional at least three and Margulis-Mohammadi \cite{MR4374970} for dimension three, they show a key feature that non-arithmetic hyperbolic manifolds satisfy. That is, they show that there are only finitely many totally geodesic hyperplanes of codimension 1. Their works rely on superrigidity theorems and are not constructive. 

In the 1980's, Gromov--Piatestski-Shapiro \cite{MR0932135} introduced a construction that produces non-arithmetic hyperbolic manifolds from two non-commensurable arithmetic manifolds. Our main result strengthens the results of Bader--Fisher--Miller--Stover \cite{MR4250391}  and Margulis-Mohammadi \cite{MR4374970} and gives a quantitative finiteness result for the number of totally geodesic hyperplanes of codimension 1 of Gromov--Piatestski-Shapiro hybrid manifolds. Throughout we assume $n\ge3$.

\begin{thm}\label{thm:count}
    Let $M$ be a non-arithmetic hyperbolic $n$-manifold obtained from the Gromov--Piatestski-Shapiro gluing construction. The number of totally geodesic hyperplanes of codimension 1 in $M$ is at most
    $$L(\vol_{n-1}(\Sigma)\vol_n(M)\eta_X^{-1}\kappa_X^{-1})^{L/\constK\label{K:count}}$$
    where $L=L(n)$, $\ref{K:count}=O(\kappa_X^2)$ and where $\kappa_X$ is the rate of decay of correlations of the geodesic flow on the tangent bundle of $M$ (see Corollary~\ref{Cor:Decay of Cor}) and $\eta_X$ is related to the Margulis constant of $M$, see Equation~\ref{eqn: Margulis constant}. 
\end{thm}

The $n=3$ case is recent work of Lindenstrauss--Mohammadi \cite{MR4549089}.  Related work of Belolipetsky--Bogachev--Kolpakov--Slavich \cite{BBKS23} prove a criterion for arithmeticity of hyperbolic manifolds depending on the number of fc-subspaces and showed an upper bound of the number of fc-subspaces that is linear in the volume in the non-arithmetic case.

The conclusion of Theorem \ref{thm:count} giving an upper bound to the  totally geodesic hyperplanes of codimension 1 applies to the examples of Raimbault \cite{10.1093/imrn/rns151} and Gelander--Levit \cite{Commclasses}. See Remark \ref{rmk:extending} for more information.

Let $d$ be a right invariant metric on the frame bundle of the universal cover of $M$  defined by using the Killing form. This induces the hyperbolic metric on $\mathbb H^{n+1}$ and the metric on $M$. This in turn gives rise to a natural volume form on $M$ and its submanifolds. Going forward we drop the subscript on volume and by ``$\vol$" we mean the one induced from the metric of the correct dimension.

The proof of Theorem~\ref{thm:count} relies on a number of ideas including a new dynamical result of effective density of periodic orbits on lattice quotients of $\SO(n,1)$ (Theorem \ref{Thm:EffDen}) and a new result on a restricted projection problem (Theorem~\ref{thm:ProjSO(n,1)}) as well as techniques involving Margulis functions, ideas of geometric measure theory (coarse dimension), and an equidistribution theorem for large measures of simple Lie groups that behave horospherically (Theorem~\ref{thm:VenkateshSO(n,1)}).

For $n\ge3$, let $\SO(n,1)$ denote the group of determinant one matrices that preserve the quadratic form $$Q_0(x_1,\ldots,x_{n+1})=2x_1x_{n+1}-\sum_{k=2} ^n x_k^2$$ and $G=\SO^\circ(n,1)$ denote the connected component of the identity. Geometrically, $G$ can be thought of as the oriented frame bundle of hyperbolic space $\mathbb H^n$.

Let $H= \textrm{Stab}_G(e_n)\simeq\SO^\circ(n-1,1)$ (see also Section 2) where $n\ge3$ and $\Gamma$ be a lattice of $G$. We are interested in the action of $H$ on $X=G/\Gamma$. Note that closed orbits of $H=\SO(n-1,1)$ in $X$ map to totally geodesic hyperplanes of codimension 1 in the hyperbolic orbifold $\mathbb H^n/\Gamma$ under the quotient map $\pi:\mathbb H^n\to \mathbb H^n/\Gamma.$ Let $X_\eta$ denote the part of $X$ where the injectivity is bounded below by $\eta$.

Perhaps surprisingly, Theorem~\ref{thm:count}  follows from Angular Rigidity of the Gromov--Piatestski-Shapiro hybrid manifolds (proved in \cite{FLMS}) and the following new effective density statement of $H$-orbits.

\begin{thm}\label{Thm:EffDen}
    Let $Y\subset X$ be a periodic $H$-orbit in $X$. Then for every $x\in X_{\vol{(Y)}^{-\constK\label{K:den}}}$ we have
    $$d_X(x,Y)\le \constE\label{C:main} \vol(Y)^{-\ref{K:den}}$$
    where $\ref{K:den}\gg\kappa_X^2 $ and $\ref{C:main}=\ref{C:main}(\kappa_X,\vol(X),\eta_X)$. Here, $\kappa_X$  is rate of the decay of correlations of the geodesic flow on $X$ and $\eta_X$ is related to the Margulis constant of $M$, see Equation~\ref{eqn: Margulis constant}. 
\end{thm}

If $\Gamma$ is a congruence lattice, then the rate $\ref{K:den}$ is absolute. Notably, we make no assumption on arithmeticity of $\Gamma$ in Theorem \ref{Thm:EffDen}. In the case that $\Gamma$ is arithmetic, Theorem \ref{Thm:EffDen} is a special case of work by Einsiedler--Margulis--Venkatesh~\cite{MR2507639} or Einsiedler--Margulis--Mohammadi--Venkatesh~\cite{MR4066475}. Furthermore, when $G=\SO(3,1)$, Theorem \ref{Thm:EffDen} was proved in Lindenstrauss--Mohammadi \cite{MR4549089}.

The main purpose of this article is to use effective density of periodic orbits to count totally geodesic subspaces of Gromov--Piatestski-Shapiro hybrid manifolds. However, it is plausible that an effective equidistribution result such as Theorem 1.1 of \cite{LMW2} is in reach for $\SO(n,1)$. See also \cite{ZL24} and \cite{LMWY25} for results using similar ideas. 

\textbf{Overview of effective density}. The proof of Theorem \ref{Thm:EffDen} will take up the majority of this paper. As such, we give an outline of the proof of Theorem \ref{Thm:EffDen}. We follow the strategy of Lindenstrauss--Mohammadi \cite{MR4549089}. 

Let $A=\{a_t = \text{diag}(e^{t},1,\ldots,1,e^{-t})\in\Mat(n+1):t\in \R\}$ be a 1-parameter family of diagonalizable elements. Let $N$ denote the unipotent subgroup that geometrically corresponds to the expanding horosphere of $a_t$. Explicitly, 
$$N:=\left\{ n(\mathbf s,r):=\left(
			\begin{matrix}
				1 & (\mathbf s,r) & \frac{1}{2}\|(\mathbf s,r)\|^2   \\
				\mathbf{0_{n-1}} ^T & \mathrm{Id}_{n-1}  & (\mathbf s,r)^T \\
				0 &  \mathbf{0_{n-1}}  & 1 
			\end{matrix}  
		\right)\in\Mat(n+1,\mathbb R):(\mathbf{s},r)\in \mathbb R^{n-1} \right\}$$
 where $\|\cdot\|$ denotes the standard Euclidean norm on $\mathbb R^{n-1}$.
  
  Let $u_\s := n(\s,0)$ denote the part of $N$ inside of $H$ and $v_r:=n(\mathbf{0},r)$ denote the remaining horospherical direction. Throughout we normalize so that $B_1 ^U:=\{u_\s:\|s\|\le1\}$ has Lebesgue measure 1.

  We write the Lie algebra of $\SO(n,1)$ as $\Ad(H)$-invariant subspaces $\text{Lie}(H)\oplus \text{Lie}(H)^\perp$ and denote the orthogonal complement with respect to the Killing form  as $\mathfrak r :=\text{Lie}(H)^\perp.$  This decomposition is  $\Ad(H)$-invariant by our choice of $H$. The vector space $\mathfrak{r}\simeq \mathbb R^{n}$ can be naturally partitioned into three types of directions: the expanding direction of $A$ where lengths of vectors increase under the $A$-action, the contracting directions of $A$ where lengths of vectors decrease under the $A$-action, and the neutral direction where the length of vectors doesn't change under the $A$-action. 

Step 1: The first step is the construction of a Margulis function $f_Y$ associated to a closed orbit $Y=Hx$. In Theorem \ref{thm:VolBound} we prove that $f_Y$ is integrable with respect to the invariant probability measure $\mu_Y$ on $Y$ with a bound comparable to $\vol(Y)$.

By using integrability of $f_Y$ and a pigeonhole argument, in Proposition \ref{prop:input}, we find a finite subset $F$ of $\mathfrak {r}$ with cardinality comparable to $\vol(Y)$ such that it has coarse dimension nearly 1 and a point $y\in Y$ such that $\exp(F)y\subset Y$. For a precise definition see Equation \ref{eqn: coarse dim} and see also Definition 5.1 of Bourgain--Furman--Lindenstrauss--Mozes \cite{BFLM}.  Since $F\subset \mathfrak r$ and $\exp(F)y\subset Y$, then we interpret this as a dimension gain in the transversal direction to $H$. In fact, for technical reasons that we omit in this overview, we must further zoom into the set $F$ so that the error terms arising from the contracting and neutral directions are negligible in Step 3. This step requires work of Bourgain--Furman--Lindenstrauss--Mozes \cite{BFLM}. See Section \ref{sec:input} for a discussion of this.

Step 2: We use a projection theorem to move the dimension gained in Step 1 into the remaining horospherical direction.

 Let $\xi_\s$ denote the projection onto the expanding direction of $\Ad(u_\s)$ applied to an element of $\mathfrak r$. That is,
$\xi_\s:\mathfrak r\simeq \mathbb R^{n}\to \mathbb R$ is given by $$\xi_\s(Z)=\big(\Ad(u_\s)Z\big)^+$$ where $\cdot^+$ denotes the projection of a vector onto the single dimension corresponding to the expanding direction. In the next section, we explicitly describe this map in a suitable coordinate system. The following projection theorem states that the set $F\subset \mathfrak{r}$ of coarse dimension $\delta$ can be projected onto the expanding direction and preserve the coarse dimension $\delta$.

\begin{thm}\label{thm:ProjSO(n,1)}
Let $0<\delta<1$ and $0<b_0< b_1<1$.
Let $F\subset B_{\mathfrak r}(0, b_1)$ be a finite set which satisfies the following   
\begin{equation}\label{eqn: coarse dim}
\frac{\#(B_{\mathfrak r}(Z, b)\cap F)}{\#F}\leq C (b/b_1)^\delta\qquad\text{for all $Z\in\mathfrak r$ and all $b\geq b_0$}     
\end{equation}

Let $0<\varepsilon<\delta/100$. There exists $B\subset B_1 ^U$ with 
\[
|B_1^U\setminus B|\leq 0.01|B_1 ^U|
\]
and for every $u_{\bf s}\in B$, there exists $F_{{\bf s}}\subset F$ with 
\[
\#(F\setminus F_{{\bf s}})\leq 0.01\cdot (\# F)
\]
such that the following holds. For all $Z\in F_{{\bf s}}$ and all $b\geq b_0$, we have 
\[
\#\{Z'\in F_{\s}: |\xi_{\bf s}(Z')-\xi_{\bf s}(Z)|\leq  b\}\leq C_\varepsilon  (b_0/b_1)^{-\varepsilon}\cdot  (b/b_1)^{\delta}\cdot (\# F)
\] 
where $C_\varepsilon$ depends on $C$, $\varepsilon$, and $n$.   
\end{thm}

As such, a large subset of $F$ has that the projection under $\xi_\mathbf s$  has coarse dimension  at all scales $b>b_1 ^{-1}(\#F)^{-1}$.  Moreover, by Step 1 and Bourgain--Furman--Lindenstrauss--Mozes \cite{BFLM}, other terms have negligible error. Informally what this has accomplished is reduced the analysis of $a_tu_s\exp(F)y$ to that of $\exp(e^t\xi_\mathbf s(F))y_1$ with errors of order $Cb_1$ for a fixed constant $C$ and where $y_1=a_tu_\s\exp(w_0)y$ and $w_0$ is in the fiber of $y$. That is to say, we have reduced the analysis to just understanding the fiber in the expanding direction. Theorem \ref{thm:ProjSO(n,1)} is proved in Section \ref{sec:proj} in a more general context.

Step 3: We can now consider the counting measure on the fiber of the expanding direction $e^{t}\xi _\mathbf s (F)$. By Step 2, we know that this set has coarse dimension $\delta$ for any $\delta\in(1/2,1)$ on sufficiently large scales. In particular, after normalizing, there is a probability measure $\rho$ on $[0,1]$ (thought of as a subset of the expanding direction) which satisfies 
$$\rho(J)\le C b^\delta$$
for every interval $J$ of length $ b_1^{-1}\# F^{-1}< b$ and a constant $C\ge 1$.  That is, we have a measure on the horosphere in $\exp(\mathfrak r)$ with nearly full dimension. The following theorem shows that expanding translates of the Lebesgue measure on $U$ and the measure constructed on the horosphere in $\exp(\mathfrak r)$ equidistribute.

\begin{thm}\label{thm:VenkateshSO(n,1)}
There exists $\constK\label{K:Venk}\gg \kappa_X ^2$ and $\varepsilon_0\gg \kappa_X ^2$ such that the following holds. Let $0<b\le 1$. Let $\rho$ be a probability measure on $[0,1]$ which satisfies $$\rho(J)\le Cb^{\delta}$$ for every interval $J$ of length $b$ and a constant $C\ge1$. Let $\eta\in(0,1)$ and $x\in X_\eta$.

If $\delta\ge1-\varepsilon_0$, then for $|\log(b)/4|\le t\le |\log(b)/2|$ and 
 $f\in C_c^\infty(X)+\mathbb C \cdot1$, we have
    $$\left|\int_{[0,1]}\int_{B_1^U}f(a_t u_{\mathbf{s}}v_{r}x)d\mathbf{s}d\rho(r)-\int_Xf(x)dm_X(x)\right|\ll C^{1/2} S(f)e^{-\ref{K:Venk} t}\eta^{-L/2}.$$ 
\end{thm}
Theorem \ref{thm:VenkateshSO(n,1)} is proved in a more general context in the Appendix in section \ref{section: Appendix A}. In particular, expanding translates of $\exp(e^t\xi_\mathbf s(F))y_1$ equidistribute in $X$. By noting that the range for the parameter $t$ depends on $b$ which is comparable to $\vol(Y)^{-1}$, we can complete the proof of Theorem \ref{Thm:EffDen}. The details can be found in Section \ref{section: proofs}.

\textbf{Acknowledgments.} The authors would like to thank Amir Mohammadi for many enlightening conversations. Additionally, they thank Zuo Lin and Pratyush Sarkar for helpful comments about an earlier version of the paper. When this work began, A.S. was supported by the National Science Foundation Postdoctoral Fellowship under grant number DMS-2103136.

%%%%%%%%%%%%%%%%%%%%%%%%%%%%%%%%%%%%%%%%%%%%%%%%%%%%%%%%%%%%%%%%%%%%%%%%%%%%%%%%%%%%%%%%%%%%%%%%%%%%%%%%%%%%%%%%%%%%%%%%%%%%%%%%%%%%%%%%%%%%%%%%%%%%%%%%%%%%%%%%%%%%%%%%%%%%%%%%%%%%%%%%%%%%%%%%%%%%%%%%%%%%%%%%%%%%%%%%%%%%%%%%%%%%%%%%%%%%%%%%%%%%%%%%%%
\section{Preliminaries}\label{Preliminaries}
In this section we fix notation.

\textbf{Asymptotic notation.} We write $A\ll B$ or $A=O(B)$ if there is some constant $C>0$ such that $A\le CB$. We will use subscripts to indicate the dependence of the constant on parameters. We write $A\asymp B$ when the ratio lies in $[C^{-1},C]$ for a constant $C\ge1$.

\textbf{Hybrid manifolds.} In this section we briefly recall the construction of Gromov--Piatestski-Shapiro \cite{MR0932135}. Let $\Gamma_1$ and $\Gamma_2$ be two torsion free lattices in $\SO(n,1)$. Then $M_i=\mathbb H^n/\Gamma_i$ are hyperbolic manifolds (i.e. complete Riemannian manifolds of constant negative curvature). We assume there exists connected submanifolds $N_i\subset M_i$ of dimension $n$ with boundary such that\begin{itemize}
    \item Every connected component of the hypersurface $\partial N_i\subset M_i$ is a totally geodesic embedded hypersurface in $M_i$ which separates $M_i$. In particular, every connected component of $\partial N_i$ is an $(n-1)$-dimensional hyperbolic manifold. 
    \item $\partial N_1$ and $\partial N_2$ are isometric.
\end{itemize}
Let $M$ be the manifold obtained by gluing $N_1$ and $N_2$ using the isometry between $N_1$ and $N_2$. Then $M$ carries a complete hyperbolic metric and, if $\Gamma_1$ and $\Gamma_2$ are arithmetic and non-commensurable, then $M$ is non-arithmetic \cite{MR0932135}.

\textbf{Lie groups, Lie algebras, and norms.} Recall, $G=\SO^\circ(n,1)$ and $H=\SO^\circ(n-1,1)$. We can think about $H$ as the group of determinant one matrices that preserve the restriction of the form $Q_0$ to $\R e_1\oplus\cdots\oplus\R e_{n-1}\oplus \{0\}\oplus \R e_{n+1}$. Fix $\H^{n-1}\subset \H^n$ with an orientation so that $H=\{g\in G:g(\mathbb H^{n-1})=\H^{n-1}\}$. Let $o\in\H^n$ be a reference point and $w_o\in \mathrm  T^1(\H^n)$ be a reference vector. Then $\H^n$ may be identified with $K\backslash G$ and $ \mathrm  T^1(\H^n)$ may be identified with $M\backslash G$ where $K$ is the stabilizer of $o$ under the action of $G$ and $M$ is the stabilizer of $w_o$ under the action of $G$. In fact, $M\simeq\SO(n)$. For $w=Mg\in \mathrm  T^1(\H^n)=M\backslash G$, let $w^{\pm}\in\partial \H^n$ denote the forward and backward endpoints of the geodesic that $w$ determines and $g^{\pm}:=w_o^{\pm}g\in \partial \H^n$.

Recall, $A=\{a_t = \text{diag}(e^{t},1,\ldots,1,e^{-t})\in\Mat(n+1):t\in \R\}$ and notice that since $M$ and $A$ commute, the right action of $A$ on $M\backslash G$ corresponds to the unit speed geodesic flow. The right action of $A$ on the frame bundle $G$ corresponds to the frame flow. Let $N$ denote the unipotent subgroup that geometrically corresponds to the expanding horosphere of $a_t = \text{diag}(e^t,1,\ldots,1,e^{-t})$. We describe $N$ explicitly after defining some notation; for $\mathbf{s}\in \mathbb R^{n-2}$ and $r\in\mathbb R$, we let $(\mathbf s,r)$ denote the vector of length $n-1$ given by concatenating $\mathbf s$ and $r$. Then
$$N:=\left\{ n(\mathbf s,r):=\left(
			\begin{matrix}
				1 & (\mathbf s,r) & \frac{1}{2}\|(\mathbf s,r)\|^2   \\
				\mathbf{0_{n-1}} ^T & \mathrm{Id}_{n-1}  & (\mathbf s,r)^T \\
				0 &  \mathbf{0_{n-1}}  & 1 
			\end{matrix}  
		\right)\in\Mat(n+1,\mathbb R):(\mathbf{s},r)\in \mathbb R^{n-1} \right\}$$
  where $\|\cdot\|$ denotes the standard Euclidean norm on $\mathbb R^{n-1}$.
Clearly, $N\simeq \mathbb R^{n-1}$. Additionally, define 
$$U:=N\cap H=\left\{ u_{\mathbf s}:=n(\mathbf s,0):\mathbf{s}\in \mathbb R^{n-2} \right\}\simeq \mathbb R^{n-2}\text{ and } V = \{v_r=n(\mathbf{0},r):r\in\mathbb R \}.$$ We have that $a_t n(\s,r)a_{-t}=n(e^{t}(\s,r))$.

Similarly, let 
  $$N^-:=\left\{ n^-(\mathbf s,r):=\left(
			\begin{matrix}
				1 & \mathbf{0_{n-1}} & 0   \\
				(\mathbf s,r)^T & \mathrm{Id}_{n-1}  & \mathbf{0_{n-1}} ^T  \\
				\frac{1}{2}\|(\mathbf s,r)\|^2  &  (\mathbf s,r)  & 1 
			\end{matrix}  
		\right)\in\Mat(n+1,\mathbb R):(\mathbf{s},r)\in \mathbb R^{n-1} \right\}$$
   where $\|\cdot\|$ denotes the standard Euclidean norm on $\mathbb R^{n-1}$.
Clearly, $N^-\simeq \mathbb R^{n-1}$. Recall that we normalize so that $B_1 ^U$ has Lebesgue measure 1. Additionally, define 
$$U^-:=N^-\cap H=\left\{ u^- _{\mathbf s}:=n^-(\mathbf s,0):\mathbf{s}\in \mathbb R^{n-2} \right\}\simeq \mathbb R^{n-2}\text{ and } V^- = \{v^- _r=n^-(\mathbf{0},r):r\in\mathbb R \}.$$

 The adjoint action of $H$ on the Lie algebra $\Lie(G)$ consists the invariant subspaces $\Lie(H)$ and the compliment that explicitly be written as 
$$\mathfrak r := \left\{Z(r_1,\mathbf c, r_2):=\left(
		\begin{matrix}
		0 & \mathbf{0_{n-2}} & r_1 &0   \\
		\mathbf{0_{n-2}}^T & \mathbf{0_{n-2}} & \mathbf{c}^T & \mathbf{0_{n-2}} ^T\\
        r_2 & -\mathbf c & 0 & r_1\\
        0 & \mathbf{0_{n-2}} & r_2 &0
		\end{matrix}  
		\right)\in\Mat(n+1,\mathbb R):r_1,r_2\in\R, \mathbf c \in \R^{n-2}\right\}.$$
  We have that $\mathfrak r$ is $n$-dimensional and we endow it with the max norm on $\mathbb R^d$. For $\beta\in(0,1)$, we define the ball about the origin of radius $\beta$ as
      $$B _\mathfrak{r} (0,\beta):=\{Z(r_1,\mathbf c, r_2)\in\mathfrak{r}:\|Z(r_1, \c, r_2)\|\le \beta \}.$$

We remark that adjoint action of $H$ on $\mathfrak{r}$ is isomorphic to the restriction of the linear action of $H$ to $e_n ^\perp\subset \mathbb R^{n+1}=\oplus_i \mathbb Re_i$. Indeed, consider the map 
\begin{equation}\label{eqn: compliment isomorphism}
    Z(r_1,\mathbf c, r_2)\mapsto r_1 e_1 + \c\cdot (e_2+\cdots+ e_{n-1})+ r_2 e_{n+1}
\end{equation}

Under this map, we have $\mathbb R e_1  = \{ Z(r_1,\mathbf{0},0):r_1\in\mathbb R\}$ is the expanding direction of $a_t$, $\mathbb R e_{n+1}  = \{Z(0,\mathbf{0},r_2):r_2\in \mathbb R\}$ is the contracting direction of $a_t$, and $\mathbb R e_2\oplus \cdots\oplus \mathbb R e_{n-1}  = \{Z(0,\mathbf{c},0):\mathbf{c}\in 
\mathbb R^{n-2}\}$ is the neutral direction of $a_t$. With respect to these coordinates we define $Z(r_1,\mathbf c, r_2)^+=r_1$.

In particular, the adjoint action on $\mathfrak r$ acts via
\begin{equation}\label{eqn:Adjointaction}
    \Ad(a_tu_\s)Z(r_1,\c,r_2)=Z(e^t(r_1+\s\cdot \c+r_2\|\s\|^2/2),\c+r_2\s,e^{-t}r_2).
\end{equation}

We will be particularly interested in the polynomial $\xi_\s:\mathfrak r\to\mathbb R$ given by
\begin{equation}\label{eqn:xi_s}
    \xi_\s(Z):= \left(\Ad(u_\s)Z(r_1,\c,r_2)\right)^+:=r_1+\s\cdot \c+r_2\|\s\|^2/2
\end{equation}
that tracks the expanding part of a vector under the Adjoint action of $u_\s$.

Note that $C_H(A) = AM'$ where $M'\simeq \SO(n-2)$.

We define a norm on $\Lie(H)$, by taking the max norm where the coordinates are given by $\Lie(N)$, $\Lie(N^-)$, $\Lie(A)$, $\Lie(M')$. By taking the maximum of the norms on $\Lie(H)$ and $\mathfrak r$, we get a norm on $\Lie(G)$ which we denote by $\|\cdot\|$. We note that $\Lie(H)$ and $\mathfrak r$ are orthogonal spaces.

For $\beta\in(0,1)$, we define the ball of radius $\beta$ about the identity in $H$ as
$$B_\beta ^H :=\{h\in H:\|h-I\|\le\beta\}.$$%$$B_\beta ^H :=\{u^- _\mathbf{s}:\|\s\|\le\beta\}\cdot \{m\in M':\|m\|\le\beta\}\cdot\{a_t:|t|\le\beta\}\cdot\{u_\mathbf{s}:\|\s\|\le\beta\}.$$

We define $B_\beta ^G:= B_\beta ^H \cdot \exp(B_\mathfrak r (0,\beta))$.

%%%%%%%%%%%%%%%%%%%%%%%%%%%%%%%%%%%%%%%%%%%%%%%%%%%%%%%%%%%%%%%
\subsection{Baker–Campbell–Hausdorff formula}

We will need the following lemma.

\begin{lem}\label{lem: normal form}
    There exists absolute constants $\beta_0<1$ and $ \constE\label{C:normalform}\ge1$ depending only on the dimension of $G$ so that the following holds. Let $0<\beta\le \beta_0$ and $w_1,w_2\in B_\mathfrak r(0,\beta)$. There are $h\in H$ and $w\in\mathfrak r$ which satisfy
    $$.5\|w_1-w_2\|\le \|w\|\le 2\|w_1-w_2\|\text{ and } \|h-I\|\le \ref{C:normalform}\beta \|w\|$$
    so that $\exp(w_1)\exp(-w_2)=h\exp(w)$.
\end{lem}

\begin{proof}
    By the open mapping theorem, there exists $\beta_0=\beta_0(\dim(G))\ll 1$ such that the map $B_{\Lie(H)}(0,\beta_0)\times B_{\mathfrak r}(0,\beta_0)\to G$ given by $(v_1,v_2)\mapsto \exp (v_1) \exp(v_2)$ is a diffeomorphism onto its image.

    Let $0<\beta\le \beta_0$ and $w_1,w_2\in B_\mathfrak r(0,\beta)$. By the Baker–Campbell–Hausdorff formula (and potentially making $\beta_0$ smaller) and using that $[w_2,w_1]=0 + [w_2-w_1,w_1]$, there exists $w'\in \Lie(G)$ with $\|w'\|\ll\beta\|w_1-w_2\|$ such that 
    $$\exp(w_1)\exp(-w_2)=\exp(w_1-w_2 + w').$$
    
     Again, by Baker–Campbell–Hausdorff formula  and potentially making $\beta_0$ smaller, there exists $(v_1,v_2)\in B_{\Lie(H)}(0,\beta_0)\times B_{\mathfrak r}(0,\beta_0)$ and $v'\in \Lie(G)$ with $\|v'\|\ll\|v_1\|\|v_2\|$ such that 
    \begin{equation}\label{eqn: liealgebraequality}
        \exp(w_1-w_2 + w') =\exp(v_1)\exp(v_2) = \exp(v_1+v_2 + v')
    \end{equation}
    where the implied constant is absolute.
    We claim $h=\exp(v_1)$ and $w = v_2$ satisfy the claims in the lemma.

    By our choice of $\beta_0$ and Equation \ref{eqn: liealgebraequality}, we have
        \begin{equation}\label{eqn: liealgebraequality2}
        w_1-w_2+w' = v_1+v_2 + v'
        \end{equation}
and since we are working with the max norm with respect to $\Lie(H)$ and $\mathfrak r$ and they are orthogonal subspaces then $w_1,w_2,v_2\in\mathfrak r$ imply that $$\|v_1\|\ll\|w'\|+\|v'\|\ll \beta\|w_1-w_2\| + \|v_1\|\|v_2\|\ll \beta\|w_1-w_2\| + \beta_0\|v_1\|.$$
This gives $\|v_1\|\ll (1-\beta_0)^{-1}\beta \|w_1-w_2\|\ll \beta \|w_1-w_2\|$.

Now using Equation \ref{eqn: liealgebraequality2}, we have $v_2 -(w_1-w_2)= -v_1 -v' +w' $ and  $\|w'\|\ll \beta\|w_1-w_2\|$, $\|v_1\|\ll \beta \|w_1-w_2\|$, and $\|v'\|\ll \beta\|v_1\|\ll \beta^2\|w_1-w_2\|$. Thus, 
 $$|\|v_2\| -\|w_1-w_2\||\le \|v_2 -(w_1-w_2)\|\ll \beta \|w_1-w_2\|$$
 and so $$(1-\beta)\|w_1-w_2\|\le \|v_2\|\le (1+\beta)\|w_1-w_2\|.$$

 By making $\beta_0$ potentially smaller, we have the desired inequality for $v_2\in\mathfrak r$. In particular,  $\beta_0\le 1/2$.

 Lastly, we note, that by using the previous bound on $\|v_2\|$ and that $\beta\le1/2$, we have $\|v_1\|\ll \beta \|w_1-w_2\|\ll\frac{\beta}{1-\beta} \|v_2\|\ll2 \beta \|v_2\|$ and we denote the implicit constant by $C'$. Note we chose $h=\exp(v_1)$. Then,  
 $$\|h-I\| \le \sum_{k=1}^\infty \|v_1\|^k /k!\ll \beta \|v_2\|e^{2C'}$$ 
 and by taking $\ref{C:normalform} = e^{2C'}$ we are done.

\end{proof}
    
Lastly, choose $\beta_1<\beta_0$ such that for all $\beta<\beta_1$ we have
\begin{equation}\label{eqn:changingballs}
    B_{\ref{C:changeinballs}^{-1}\beta} ^G\subseteq \{g\in G:\|g-I\|<\beta\}\subset B_{\ref{C:changeinballs}\beta} ^G
\end{equation}
for some absolute $\constE\label{C:changeinballs}=\ref{C:changeinballs}(\dim(G)).$ We remark that $\|g-I\|\asymp d_G(e,g)$ where the implicit constant depends on $\dim(G)$.

%%%%%%%%%%%%%%%%%%%%%%%%%%%%%%%%%%%%%%%%%%%%%%%%%%%%%%%%%%%%%%%%%%%%%%%%%%%%%%%%%%%%%%%%%%%%%%%%%%%%%%%%%%%%%%%%%%%%%%%%%%%%%%%%%%%%%%%%%%%%%%%%%%%%%%%%%%%

\section{Thick-thin decomposition}
In this section we recall a number of standard facts of hyperbolic geometry. See for example, Bowditch \cite{MR1218098}.

We recall the thick-thin decomposition.  In short, this is a decomposition of the space $X$ as $X_{\eta}\cup X_{\eta} ^c$ where $X_{\eta}=\{x\in X: \inj(x)\ge\eta\}$ where for any $\eta>0$ small enough we have $X - X_{\eta}$ is contained in finitely many disjoint cuspidal ends given by horospheres in $X$.

More precisely, let $\xi_1,\ldots,\xi_q\in\partial \H^n$ be a finite set of $\Gamma$-representatives and fix a set of $g_1,\ldots,g_q\in G$ with $g_i^- =\xi_i$. For $T>0$, let
$$\mathcal H_i (T):= \bigcup_{t>T} Ka_{-t}Ng_i$$
denote a horosphere of depth $T$ in $G$ and 
$$\mathcal X_i(T):=\mathcal H_i(T)\Gamma$$
denote the image of the horosphere in $X$.
By Bowditch \cite{MR1218098}, there exists $T_0\ge1$ such that $\mathcal X_1(T_0),\ldots, \mathcal X_q(T_0)$ are disjoint. Moreover, if $\mathcal X(T):= \mathcal X_1(T)\sqcup\cdots
\sqcup\mathcal X_q(T)$, then for any $T\ge T_0$, the complement $X - \mathcal X(T)$ is a compact subset. This results in the \emph{thick-thin} decomposition $X = \mathcal X(T) \sqcup (X - \mathcal X(T))$ and will allows us to talk about ``how far in the cusp" a point $x\in X$ is.

 The thick-thin decomposition is related to another natural quantity on $X$. For $x\in X$, the \emph{injectivity radius} of a point $x$, denoted by $\inj(x)$, is given by the supremum over all $\eta>0$ for which the projection map $g\mapsto gx$ from $G$ to $X$ is injective on $B^G _\eta$. That is, the injectivity radius $\inj(x)$ of a point $x$ has that $B_{\inj(x)} ^G$ embeds balls from $G$ into $X$. By taking a minimum, we assume that the injectivity radius is smaller than 1.
   By Kelmer-Oh \cite{MR4330023}, we have $\inj(x)\asymp e^{-t}$ for $x \in \mathcal X_i(T)$ (and so $t>T$). Thus, choosing $T\ge T_0$ large enough so that the thick-thin decomposition holds amounts to choosing $C\eta=e^{-T}$ sufficiently small (where $C$ is the implicit constant in $\inj(x)\asymp e^{-t}$) so that $\mathcal X(T)=\mathcal X(\log((C\eta)^{-1}))=\{x\in X: \inj(x)< C\eta\}$. Let $\eta'>0$ denote the supremum of such $C\eta$ and define \begin{equation}\label{eqn: Margulis constant}
       \eta_X=\eta'/2
   \end{equation}
   We assume that $\eta_X<1.$ 
   Then $X_{\eta_X}:=\{x\in X: \inj(x)\ge\eta_X\}$ is a compact subset and we have that $X-X_{\eta_X}$ is a union of disjoint horospheres in $X$. Denote by $T_{\eta_X}\ge T_0$ the corresponding value so that $X-X_{\eta_X} = \mathcal X_1(T_{\eta_X})\sqcup \cdots \sqcup \mathcal X_q(T_{\eta_X}).$ 

Lastly, we introduce height functions which informally give us a way of thinking about ``how far in the cusp" a point $x\in X$ is. To this end, we consider the linear action of $G$ on $\mathbb R^{n+1}$ with the Euclidean norm. Let $v_i:=g_i ^{-1}e_1\in Ge_1$ where $g_1,\ldots,g_q\in G$ satisfy $g_i^- =\xi_i$ and $\|v_i\|=1$. By Mohammadi-Oh \cite{MR4556221} (in 3 dimensions) and Tamam-Warren \cite{MR4611749} (in dimensions 3 and higher), we have that the orbit $\Gamma v_i$ is a discrete subset of $\mathbb R^{n+1}$. By the definition of $\mathcal X_i$, we have if  $g\Gamma\in \mathcal X_i(T_{\eta_X})$, then there exists $k\in K$, $t>T_{\eta_X}$, $u\in N$, and $\gamma\in \Gamma$ such that
$$\|g\gamma v_i\|=\|ka_{-t}ug_i\gamma v_i\|=\|ka_{-t}ue_1\|=e^{-t}.$$
In particular, we have the following key observation (see also Equation (33) of \cite{MR4611749}): $$x=g\Gamma \in \mathcal X_i(T_{\eta_X})\text{ if and only if there exists }\gamma\in\Gamma\text{ with }\|g\gamma v_i\|\le e^{-T_{\eta_X}}.$$ By the discreteness of $\Gamma v_i$, such a $\gamma$ is unique.

Additionally, by Lemma 6.5 of Mohammadi-Oh \cite{MR4556221}  or Equation (34) of Tamam-Warren \cite{MR4611749}, there exists a lowerbound $\eta_0 = \eta(\Gamma)>1/2$ such that for any $i=1,\ldots,q$, if $g\Gamma \notin \mathcal X_i(T_{\eta_X})$, then for any $\gamma\in \Gamma$, 
\begin{equation}\label{LBonThick}
    \|g\gamma v_i\|>\eta_0.
\end{equation}

Motivated by the above, we define the following which  gives the gives a sense of ``how far in the cusp" a point is.

\begin{defn}
    For $x=g\Gamma\in X=G/\Gamma$, we define the \emph{height} of $x$ with the function  $h:X\to [2,\infty)$, where
    $$h(x):=\max_{\gamma\in\Gamma}\{\|g\gamma v_1\|^{-1},\ldots,\|g\gamma v_q\|^{-1},2\}$$
\end{defn}

If $\Gamma$ is cocompact, then we define $h:X\to [2,\infty)$ to be identically 2.

On the other hand, by Lemmas 6.5 and 6.6 of Mohammadi-Oh \cite{MR4556221} (whose proofs immediately generalize for $n>3$), we have the following.
\begin{prop}\label{prop: rank1} Let $x=g\Gamma\in X$.
    If  $h(x)\ge 1/\eta_0$, then there exists a unique $1\le j_0\le q$ and $\gamma\in \Gamma$ such that $h(x)=\|g\gamma_0v_{j_0}\|^{-1}$ and $\|g\gamma v_j\|>\eta_0$ for any $(\gamma,j)\ne (\gamma_0,j_0)$.
\end{prop}

This is the ``rank 1" phenomena of $\SO(n,1);$ a point in the thin part cannot be in two cuspidal parts simultaneously.

The following relates the height of a point with it's injectivity radius. See the proof of Proposition 6.7 of Mohammadi-Oh \cite{MR4556221}.
\begin{prop}\label{Prop:comparison}
    There exists a constant $\sigma=\sigma(\Gamma)\ge1$ such that for any $x\in X$, 
$$    \sigma^{-1}\inj(x)\le h(x)^{-1}\le \sigma \inj(x).$$
\end{prop}

\begin{remark}\label{Rem:LogCty}
The function $h$ is log continuous; that is, for any compact $K\subset G$, there exists $\sigma_h=\sigma_h(K)>1$ such that
\begin{equation}\label{logctyheight}
\sigma_h ^{-1}h(x)\le h(kx)\le \sigma_h h(x)    
\end{equation}
for all $k\in K$ and $x\in X$. Indeed, this is essentially the observation that multiplication by elements of $G$ is a continuous and invertible map. Once we have log continuity of $h$, then Equation \ref{Prop:comparison} implies log continuity of $\inj$.
\end{remark}

%%%%%%%%%%%%%%%%%%%%%%%%%%%%%%%%%%%%%%%%%%%%%%%%%%%%%%%%%%%%%%%%%%%%%%%%%%%%%%%%%%%%%%%%%%%%%%%%%%%%%%%%%%%%%%%%%%%%%%%%%%%%%%%%%%%%%%%%%%%%%%%%%%%%%%%%%%%%%%%%%%%%%%%%%%%%%%%%%%%%%%%%%%%%%%%%%%%%%%%%%%%%%%%%%%%%%%%%%%%%%%%%%%%%%%%%%%%%%%%%%%%%%%%%%%
\section{Linear algebra lemma}

In this section, we prove a linear algebra lemma that shows that on average the action of expanding horospheres increase the length of vectors in $\mathfrak r$. This behavior is well studied see e.g. Shah \cite{MR1403756}, however, we reprove this to guarantee certain exponents are large enough for later sections. See also Eskin-- Margulis--Mozes \cite{MR1609447} where they prove a similar qualitative result for expanding circles and demonstrate that in dimensions larger than three the exponent $\delta$ can be close to two.

\begin{lem}\label{Lem:LA1}
    Suppose $1/2<\delta<1$, $Z(r_1,\c,r_2)\ne 0$, and $t>0$. Then
    $$\int_{B^U_1}\frac{d\s}{\|\Ad(a_tu_\s)Z(r_1,\c,r_2)\|^\delta}\ll e^{t(-1+\delta)}\|Z(r_1,\c,r_2)\|^{-\delta}.$$
\end{lem}

We will frequently use the following inequality that estimates the measure of a sub-level set of a polynomial. Such estimates have seen many applications in dynamics, notably in Dani--Margulis \cite{MR1237827} and Kleinbock--Margulis \cite{MR1652916} where they were called $(C,\alpha)$-good functions. In fact, such estimates go back to Remez \cite{Remez}. Below we use the formulation of Katz \cite{katz}, Corollary 4.2 of the Remez inequality.

\begin{prop}[Remez' Inequality]\label{Remez}
Let $B\subset \mathbb R^n$ be a non-empty convex subset and $f:B\to\mathbb R ^m$ be a polynomial of degree $d$. For any $\varepsilon>0$, one has
$$\big|\{x\in B: \|f\|<\varepsilon\}\big|\le 4n\left(\frac{\varepsilon}{\sup_{x\in B}\|f(x)\|}\right)^{1/d}|B|$$
   where $|\cdot|$ denotes the Lebesgue measure on $\mathbb R^n$. 
\end{prop}

Recall from Equation \eqref{eqn:Adjointaction}, the adjoint action on $\mathfrak r$ acts via
$$    \Ad(a_tu_\s)Z(r_1,\c,r_2)=Z(e^t(r_1+\s\cdot \c+r_2\|\s\|^2/2),\c+r_2\s,e^{-t}r_2)
$$

\begin{proof}[Proof of Lemma \ref{Lem:LA1}]
    Without loss of generality, we suppose that $Z(r_1,\c,r_2)$ has norm 1.

\begin{itemize}
\item

Consider $|r_2|\ge 1/1000$.

We compute $$\int_{B_1 ^U}\frac{d\s}{\|\Ad(a_tu_\s)Z(r_1,\c,r_2)\|^\delta}$$ by breaking it into 
$$D_1 = \{\s\in B_1 ^U:e^t(r_1+\s \cdot\c+r_2\|\s\|^2/2))<e^{-t}/1000\}$$ and the compliment $D_2$. 

On $D_1$ our integral becomes
$$\int_{D_1}\frac{d\s}{(e^{-t}|r_2|)^\delta}\le \frac{m(D_1)}{e^{-\delta t}/1000^\delta}.$$
Proposition \ref{Remez} says $|D_1|\ll e^{-t}$ and so 

$$\int_{D_1}\frac{d\s}{\|\Ad(a_tu_\s)Z(r_1,\c,r_2)\|^\delta}\ll e^{t(-1+\delta)}.$$

For the integral on $D_2$, we consider the regions $$J_k=\{\s\in B_1 ^U:e^{-2t+k}\le|r_1+\s \cdot\c+r_2\|\s\|^2/2)|<e^{-2t+k+1}\}.$$ Then, by utilizing the lower bounds on the max imposed by the region $J_k$ and Proposition \ref{Remez} we get
\begin{align*}
\int_{D_2}\frac{d\s}{\|\Ad(a_tu_\s)Z(r_1,\c,r_2)\|^\delta} &= \sum_{k=0} ^\infty\int_{D_2\cap J_k}\frac{d\s}{\|\Ad(a_tu_\s)Z(r_1,\c,r_2)\|^\delta} \\
&\le \sum_{k=0} ^\infty\frac{|J_k\cap D_2|}{(e^{t}e^{-2t+k})^\delta} \\
&\le  \sum_{k=0} ^\infty\frac{e^{(-2t+k+1)/2}}{(e^{-t+k})^\delta} \ll e^{t(-1+\delta)}\sum_{k=0} ^\infty e^{k((1/2)-\delta)}\\
&\ll e^{t(-1+\delta)}.
\end{align*}

\item

Consider $|r_2|\le 1/1000$.

Notice here that we never need to consider the third term $e^{-t}r_2$ because the upper bound on $r_2$. Indeed, it is always beaten by the first term (so long as $r_1+\s \c+r_2\|\s\|^2/2\ne 0$).

Assume that $\|\c\|<1/100$ i.e. is very small. Then, since $r_2$ is also small, then we must have $r_1=1$ by our condition on the max norm of $Z(r_1,\c,r_2)$. Hence, $r_1+\s \c+r_2\|\s\|^2/2$, even in the worst case of $r_2$ and $\c$ being small, is at least 1/10. This is independent of $\s$.

In particular, $\|\Ad(a_tu_\s)Z(r_1,\c,r_2)\| = r_1+\s \c+r_2\|\s\|^2/2>e^t/10$ and so

$$\int_{B_1 ^U}\frac{d\s}{\|\Ad(a_tu_\s)Z(r_1,\c,r_2)\|^\delta}\le e^{-\delta t}10^\delta\ll e^{t(-1+\delta)}$$
 for $\delta>1/2$.

 Lastly, assume $\|\c\|\ge 1/100$. Then
 $$1/100\le \|\c\|\le \|\c+r_2\s\|+|r_2|\le\|\c+r_2\s\|+1/1000$$
yields that the middle term has the bound $\|\c+r_2\s\|\ge1/200 $. This is independent of $\s\in B_1 ^U$.

Consider $D_1 = \{\s:e^t(r_1+\s \c+r_2\|\s\|^2/2))<1/200\}$ and the compliment $D_2$.

On $D_1$ our integral becomes
$$\int_{D_1}\frac{d\s}{(\|\c+r_2\s\|)^\delta}\le \frac{|D_1|}{200^\delta}.$$
Proposition \ref{Remez} says $|D_1|\ll e^{-t/2}$ and so 

$$\int_{D_1}\frac{d\s}{\|\Ad(a_tu_\s)Z(r_1,\c,r_2)\|^\delta}\ll e^{t(-1+\delta)}$$
whenever $\delta>1/2.$

For the integral on $D_2$, we consider the regions $$J_k=\{\s:e^{-t+k}\le|r_1+\s \c+r_2\|\s\|^2/2)|<e^{-t+k+1}\}.$$ Then, by utilizing the lower bounds on the max imposed by the region $J_k$ and Proposition \ref{Remez} we get
\begin{align*}
\int_{D_2}\frac{d\s}{\|\Ad(a_tu_\s)Z(r_1,\c,r_2)\|^\delta} &= \sum_{k=0} ^\infty\int_{D_2\cap J_k}\frac{d\s}{\|\Ad(a_tu_\s)Z(r_1,\c,r_2)\|^\delta} \\
&\le \sum_{k=0} ^\infty\frac{|J_k\cap D_2|}{(e^te^{-t+k})^\delta} \\
&\le \sum_{k=0} ^\infty\frac{e^{(-t+k+1)/2}}{(e^{k})^\delta} \ll e^{-t/2}\sum_{k=0} ^\infty e^{k((1/2)-\delta)}\\
&\ll e^{t(-1+\delta)}
\end{align*}
whenever $\delta>1/2$.
\end{itemize}
\end{proof}

We will need to extend this to all vectors in $Ge_1$. This follows as in Lemma 5.12 of Mohammadi-Oh \cite{MR4556221} using the fact that vectors in $Ge_1$ are projectively far from the $H$-invariant line $\mathbb Re_n$.  

\begin{lem}\label{Lem:Lowerbound}
    For any vector $v\in Ge_1$, $\|v\|\ll\|v_1\|$, where $v_1$ is the projection of $v$ to $(\mathbb R e_{n}) ^\perp$ where the implied constant is absolute.
\end{lem}
\begin{proof}
    The proof is the same as in Lemma 5.12 by replacing $e_3$ with $e_n$.
\end{proof}

We can now extend our linear algebra to $Ge_1$ which corresponds to the thin parts of $X$.

\begin{lem}\label{Lem:LA}
    Suppose $1/2<\delta<1$, $v\in Ge_1$, and $t>0$. Then
    $$\int_{B^U_1}\frac{d\s}{\|a_tu_\s v\|^\delta}\ll e^{t(-1+\delta)}\|v\|^{-\delta}.$$
\end{lem}

\begin{proof}
    We can write any $v\in Ge_1\subset \mathbb R^{n+1}$ as $v=v_1+ce_n$ where $v_1\in (\mathbb R e_{n}) ^\perp$ and $c\in\mathbb R$. Then since $e_n$ is $H$-invariant, we have $a_tu_\s v= a_tu_\s v_1 + c e_n$ for any $t\in \mathbb R$ and $\s\in B_1 ^U$. Then
$$ \int_{B_1 ^U}\frac{d\s}{\|a_t u_\s v\|^\delta}= \int_{B_1 ^U}\frac{d\s}{(\| a_t u_\s v_1\|^2+c^2)^{\delta/2} }
        \le\int_{B_1 ^U}\frac{d\s}{\| a_t u_\s v_1\|^\delta}
        \ll e^{t(-1+\delta)}\|v_1\|^{-\delta}$$
    where the last comparison comes from Lemma \ref{Lem:LA1}. By Lemma \ref{Lem:Lowerbound} we have  $\|v_1\|\gg \|v\|$, which gives the desired inequality.
\end{proof}

Additionally, we need the following corollary that discretizes the previous lemma.

\begin{cor}\label{Cor:LA}
    For any $\delta\in(1/2,1)$, there exists $m_\delta>0$ such that for any $v\in Ge_1\setminus\{0\}$, we have
    $$\int_{B^U _1}\frac{d\s}{\|a_{m_\delta}u_\s v\|^\delta}\le e^{-1}\sigma^{-2} \|v\|^{-\delta}.$$
\end{cor}
\begin{proof}
    Choose $m_\delta>0$ such that $Ce^{m_\delta(-1+\delta)}=e^{-1}/\sigma^2$  where $C$ is the implicit constant from Lemma \ref{Lem:LA} and $\sigma$ is the factor from Proposition \ref{Prop:comparison}. Then, by Lemma \ref{Lem:LA}, we have 
    $$\int_{B^U _1}\frac{d\s}{\|a_{m_\delta}u_\s v\|^\delta}\le e^{-1}\sigma^{-2} \|v\|^{-\delta}.$$

    We note that $m_\delta\to\infty$ as $\delta\to1.$
\end{proof}

\begin{cor}\label{Cor:LAInj}
    For any $\delta\in(1/2,1)$ $x\in X$ we have
$$\int_{B_1 ^U} \inj(a_{m_\delta}u_\s y)^{-\delta} d\s \le e^{-1}\inj(y)^{-\delta} + O(1)$$
where the implicit constant depends only on $X$.
\end{cor}
\begin{proof}
    If $a_{m_\delta}u_\s y\in X_{\eta_X}$, then the injectivity radius is bounded below by an absolute constant and so  $ \inj(a_{m_\delta}u_\s y)^{-\delta}$ is bounded above by an absolute constant.

    Otherwise, $a_{m_\delta}u_\s y = a_{m_\delta}u_\s g\Gamma\in X- X_{\eta_X}$ and hence there exists some $v\in Ge_1\setminus\{0\}$ with $h(a_{m_\delta}u_\s y) = \|a_{m_\delta}u_\s v\|^{-\delta}$ and, by Proposition \ref{Prop:comparison}, this is comparable to $\inj(a_{m_\delta}u_\s y)^{-\delta}$.
    
    Hence, by combining the above and using Corollary \ref{Cor:LA}, we deduce
    \begin{align*}\int_{B_1 ^U} \inj(a_{m_\delta}u_\s y)^{-\delta} d\s &\le\sigma \int h(a_{m_\delta}u_\s y)^\delta d\s =\int_{B_1 ^U}\|a_{m_\delta}u_\s v\|^{-\delta}d\s\\
    &=\sigma^{-1}e^{-1}\|v\|^{-\delta}  \le e^{-1}\inj(y)^{-\delta} .
    \end{align*}

    Hence, in general, we have $$\inj(a_{m_\delta}u_\s y)^{-\delta} d\s \le e^{-1}\inj(y)^{-\delta} + O(1)$$

\end{proof}

%%%%%%%%%%%%%%%%%%%%%%%%%%%%%%%%%%%%%%%%%%%%%%%%%%%%%%%%%%%%%%%%%%%%%%%%%%%%%%%%%%%%%%%%%%%%%%%%%%%%%%%%%%%%%%%%%%%%%%

\section{Quantitative non-divergence and a return lemma}

In this section we use the linearization from the previous section to prove a quantitative non-divergence result.

\begin{lem}\label{lem:returnrepresentation}
    Let $\eta>0$ and let $B\subseteq B^U _1$ be a ball of Lebesgue measure at least $\eta$ and $v\in\mathbb R^{n+1}$. There exists $\constE\label{C:returnrepresentation}=\ref{C:returnrepresentation}(\dim(G))$ so that for $\rho>0$ and $v\in G e_1$ we have
    $$|\{\s\in B:\|a_tu_\s v\|\le e^t\eta^2\|v\|\rho^2\}|\le \ref{C:returnrepresentation} \rho |B|.$$
\end{lem}

\begin{proof}
    We note that $\rho$ is chosen to be small compared to absolute constants.
    
    Write $v= v_1+c_ve_n$ where $c_v\in \mathbb R$ and $v_1\in e_n^\perp$. Note $Q_0(e_1)=0$ and so for any $g\in G$, we have $Q_0(ge_1)=0$ and, in particular, $Q(v)=0.$ By Lemma $\ref{Lem:Lowerbound}$, we have $\|v_1\|\ge c\|v\|$ for some absolute constant $c\in(0,1).$ Note for $h\in H$, $hv=ce_{n}+hv_1$ and recall that we can identify $e_n^\perp$ with $\mathfrak r$ by Equation \ref{eqn:  compliment isomorphism} and so we can refer to $v_1$ as $Z=Z(r_1,\c,r_2)\in\mathfrak r$. 
    By Shah \cite{MR1403756}, the polynomial function $\xi_\s(Z)$ (see Equation \ref{eqn:xi_s}) satisfies
    $$\sup_{\s\in B}|\xi_\s(Z)|\ge 16k\eta^2\|Z\|$$
    for some absolute $k>0.$
    
    Let
    $$B(Z,\alpha)=\{\s\in B:|\xi_\s(Z)|\le k\alpha \eta^2\|Z\|\}$$

    By Proposition \ref{Remez}, $|B(Z,\alpha)|\le n \alpha^{1/2}|B|$.

    Now let $\alpha = c^{-1}k^{-1}\rho^2$ where we choose $\rho$ shortly.  Then we conclude $$|B(v_1,\alpha)|\le n\alpha^{1/2}|B|=nc^{-1/2}\rho|B|.$$

    Now we choose $\rho$ so that $\alpha\le 1$ and $nc^{-1/2}\rho<1.$

    Let $\s \in B -B(v_1,
    \alpha)$ so that
    $$|\xi_\s(v_1)|\ge kc^{-1}\eta^2\rho^2\|v_1\|.$$

    Then by noting that $a_t$ expands the first term of $v_1 = Z(r_1,\c,r_2)$ by $e^t$, we see
    \begin{align*}
        e^t\eta^2\rho^2\|v\|&\le e^tc^{-1}\eta^2\rho^2\|v_1\|\\
        &\le e^t|\xi_\s(v_1)|\\
&        \le\|a_tu_\s v_1\|\le \|a_tu_\s v\|
    \end{align*}

    so that the claim holds by taking $\ref{C:returnrepresentation}=nk^{-1}c^{-1/2}$.
\end{proof}

This allows us to show that points return to the thick part after using $U$ to move them into a general position where they can be expanded by the diagonal. 

\begin{prop} \label{prop:nondiv}
    There exists $\constE\label{C:nondiv}=\ref{C:nondiv}(\dim(G), \ref{C:returnrepresentation})\ge1$ with the following property. Let $0<\alpha<1$, $0<\eta\le1$, and $x\in X$. Let $B\subseteq B^U _1$ be ball with Lebesgue measure at least $\eta$. Then 
    $$|\{\s\in B:\inj(a_t u_\s x)<\alpha^2\}|<\ref{C:nondiv}\alpha
    |B|$$
    so long as $t\ge |\log(\eta^2\inj(x))|+\ref{C:nondiv}$.
\end{prop}
\begin{proof}
    Let $x=g\Gamma\in X$. Let $\gamma\in\Gamma$ and $1\le i\le q$. 
    By Proposition \ref{Prop:comparison}, it suffices to work with the height function. By Proposition \ref{lem:returnrepresentation}, 
    \begin{equation}\label{eqn: prop3.3conseq}
    |\{\s\in B:\|a_tu_s g\gamma v_i\|\le e^t\eta^2\rho^2\|g\gamma v_i\|\}|\le \ref{C:returnrepresentation}\rho|B|    
    \end{equation}
for every $\rho\in(0,1)$. Let $\rho_0=.1\ref{C:returnrepresentation}^{-1}$. We will use the value $\rho=\rho_0$ in Equation \eqref{eqn: prop3.3conseq}. As such, by the above and Proposition \ref{Prop:comparison},  there exists some $\s\in B$ so that 
$$\|a_tu_\s g\gamma v_i\|\ge e^t\eta^2\rho_0^2\|g\gamma v_i\|\ge e^t \eta^2\rho_0 ^2\sigma^{-1}\inj(x).$$

  Now let $t\ge |\log(\eta^2\inj(x))|+\ref{C:nondiv}$  so the previous equation simplifies to
$$\|a_tu_\s g\gamma v_i\|\ge e^{\ref{C:nondiv}}\rho_0^2\sigma^{-1}.$$
This implies that for  $\ref{C:nondiv}\ge \max\{\ln(\sigma),4n\rho_0 ^{-1}\}$  we have
$$\|a_tu_\s g\gamma v_i\|\ge \rho_0^2$$
and so $\sup_{\s\in B}h(a_tu_\s x)^{1/2}\le \rho_0^{-1}$.

Then by Proposition \ref{Remez} and our choice of $\ref{C:nondiv}$ we have,
    $$|\{\s\in B:\inj(a_t u_\s x)<\alpha^2\}|<4n\rho_0 ^{-1}\alpha
    |B|\le \ref{C:nondiv}\alpha|B|.$$\end{proof}
As a corollary, we can guarantee returns to the thick part. See also Mohammadi-Oh \cite{MR4556221} for the case $G=\SO(3,1)$.
%%%%%%%%%%%%%%%%%%%%%%%%%%%%%%%%%%%%%%
\begin{cor}\label{cor:returnlemma}
    There exists $\eta_X\in(0,1)$ depending on $X$ so that the following holds. For every $x\in X$, there exists $\s\in B_1 ^U$ so that $a_tu_\s\in X_{\eta_X}$ where $t=|\log(\inj(x))|+\ref{C:nondiv}.$

\end{cor}
\begin{proof}
 Apply the previous proposition with $B=B_1 ^U$, $\eta=1$,  and $\alpha\le.01\ref{C:nondiv}^{-1}$. Then, by potentially making $\eta_X\le \alpha^2$ we have
 $$|\{s\in B:a_tu_\s x \in X_{\eta_X}\}|\ge.99|B|.$$
\end{proof}

%%%%%%%%%%%%%%%%%%%%%%%%%%%%%%%%%%%%%%%%%%%%%%%%%%%%%%%%%%%%%%%%%%%%%%%%%%%%%%%%%%%%%%%%%%%%%%%%%%%%%%%%%%%%%%%%%%%%%%%%%%%%%%%%%%%%
%%%%%%%%%%%%%%%%%%%%%%%%%%%%%%%%%%%%%%%%%%%%%%%%%
\section{Sheets and Margulis functions}

In this section, we consider constructions that will eventually allow us to get one-dimensional energy as in Proposition \ref{prop:input}. Similar ideas were used in Mohammadi-Oh \cite{MR4556221} and  Lindenstrauss--Mohammadi \cite{MR4549089} for $\SL(2,\mathbb C)$ and $\SL(2,\mathbb R)\times \SL(2,\mathbb R)$ and \cite{SS22} for certain semisimple Lie groups. While $\SO(n,1)$ falls under the framework of \cite{SS22}, we have a concrete application in mind and so we need a strengthening of certain constants that appear in \cite{SS22}. As such, we include certain proofs from \cite{SS22} for completeness.

%%%%%%%%%%%%%%%%%%%%%%%sheets
\begin{remark}In this remark we collect a few constants that will be useful in the following arguments. 

By shrinking $\eta_X$ potentially, we have that since the exponential map $\Lie(G)\to G$ is a local diffeomorphism, there exists some absolute constant $\sigma_0>1$ so that for all $w\in B ^{\Lie(G)} _{\eta_X} (0)$ and $x\in X$,
$$\sigma_0 ^{-1}\|w\|\le d(x,\exp(w)x)\le \sigma_0 \|w\|.$$

Since $B_1 ^U$ is compact, there exists some absolute constant $\sigma_U>1$ such that $\|\Ad(u)v\|\le\sigma_U\|v\|$ for all $v\in \Lie(G)$ and $u\in B_1 ^U$.
\end{remark}
%%%%%%%%%%%%%%%%%%%%%%%%%%%%%%%%%%%%%%%%%%%%%%%%%%%%%%%%%%%%%%%%%%%%%%%%%%%%%%%%%%%%%

\begin{defn} \label{def: sheets}

Let $\varepsilon\in(0,1)$. Recall the decomposition $\Lie(G)=\Lie(H)\oplus \mathfrak r$. For $Y$ a closed $H$-orbit and $x\in X$, define
$$I_Y(x,\varepsilon)=\{w\in\mathfrak{r}:0<\|w\|<\varepsilon \inj(x), \exp(w)x\in Y\}.$$

We will work with $I_Y(x):=I_Y(x,\varepsilon_0)$ where $\varepsilon_0:=\min\{e^{-\kappa\sigma_U-\ref{C:nondiv}}\eta_X,\frac{\sigma_0 ^{-1}}{2},\frac{\beta_0}{2}\}$ and $\kappa$ satisfies $\sigma_Ue^{-\kappa \sigma_U}<\frac{\sigma_0 ^{-1}}{2}$.
\end{defn}

The following proposition, which is essentially Proposition 25 of \cite{SS22} applied to $G=\SO(n,1)$, shows that the number of nearby sheets is controlled by the volume of $Y$. We include the proof for completeness. See also Lemma 9.2 of Lindenstrauss--Mohammadi \cite{MR4549089} and Lemma 8.13 of Mohammadi-Oh \cite{MR4556221} for special cases.

%%%%%%%%%%%%%%%%%%%%%%%%%%%%%%%%%%%%%%%%%%%%%%
\begin{lem}\label{Lem:SheetsBound}
    There exists an absolute constant $\constE\label{C:sheets}$ such that $$\#I_Y(x)\le \ref{C:sheets}\vol(Y)$$ for every $x\in X.$
\end{lem}

\begin{proof}

We first get a preliminary bound for points $\{x\in X|\inj(x)>Q\}$.

Let $\varepsilon>0$ and $w\in I_Y(x,\varepsilon)$. Then, 
$$d(x,\exp(w)x)\le \sigma_0 \|w\|\le \sigma_0 \varepsilon_0\inj(x)<\frac{1}{2}\inj(x)$$
for any $\varepsilon<\frac{\sigma_0 ^{-1}}{2}$. Hence, $\inj(\exp(w)x)\ge \frac{\inj(x)}{4}$ and so the balls 
$$(B_Y(\exp(w)x,\inj(x)/4))_{w\in I_Y(x)}$$
are disjoint by using that we have a local product structure at small enough scales. Thus,
$$\#I_Y(x,\varepsilon)\cdot \vol(B_H(e,\inj(x)/4))\le \vol(\sqcup_{w\in I_Y(x)} B_Y(\exp(w)x,\inj(x)/4))\le \vol (Y).$$
We obtain,
\begin{align*}
    \#I_Y(x,\varepsilon)&\le \vol(B_H(e,\inj(x)/4))^{-1}\cdot\vol(Y)\\
    &= 4^{\dim(H)}\inj(x)^{-\dim(H)}\cdot\vol(Y)
\end{align*} 
 In particular, by choosing $Q$ small enough we have $X_{\eta_X}\subset \{x\in X|\inj(x)>Q\}$ and so for any $x\in X_{\eta_X}$, we have that the injectivity radius is bounded from below and we obtain the preliminary bound on the number of sheets,
\begin{equation}\label{SheetsonCpt}
    \#I_Y(x,\varepsilon)\ll \vol(Y)
\end{equation}
for any $\varepsilon<\frac{\sigma_0 ^{-1}}{2}$.

Suppose $x\in X-X_{\eta_X}$. By the Return Lemma (Corollary \ref{cor:returnlemma}), we have some $\mathbf{s}\in B_1 ^U$ such that $a_tu_{\mathbf{s}}x\in X_{\eta_X}$ where $t=|\log(\inj(x))|+\ref{C:nondiv}$. Hence, for all $\|w\|<\varepsilon_0\inj(x)$ we have
$$\|\Ad(a_t u_{\mathbf{s}})w\|\le\sigma_U e^{t}\|w\|=\sigma_Ue^{\ref{C:nondiv}}\inj(x)^{-1}\|w\|\le\sigma_Ue^{-\kappa \sigma_U}\inj(a_tu_{\mathbf{s}}x)$$
where we have used $\varepsilon_0\le e^{-\kappa\sigma_U-\ref{C:nondiv}}\eta_X\le e^{-\kappa\sigma_U-\ref{C:nondiv}}\inj(a_tu_{\mathbf{s}}x).$
Since $\kappa$ satisfies $\sigma_Ue^{-\kappa \sigma_U}<\frac{\sigma_0 ^{-1}}{2}$ and we have $Y$ is invariant under $H$, then we have shown for any $w\in I(x)$, that $\Ad(a_tu_{\mathbf{s}})w\in I(a_tu_{\mathbf{s}}x,\varepsilon)$ for some $\varepsilon<\frac{\sigma_0 ^{-1}}{2}$.

In particular, the map $w\mapsto a_tu_{\mathbf{s}}w$ from $I(x)$ to $I(a_tu_{\mathbf{s}}x,\varepsilon)$ is injective for some $\varepsilon<\frac{\sigma_0 ^{-1}}{2}$. Hence, for any $x\in X-X_{\eta_X}$
$$\#I(x)\le \#I (a_tu_{\mathbf{s}}x,\varepsilon)\ll \vol(Y)$$
where the last inequality is from Equation (\ref{SheetsonCpt}) and the fact that $a_tu_{\mathbf{s}}x\in X_{\eta_X}$. We define $\ref{C:sheets}$ to be the max of the implicit constant from this last equation and the implicit constant from Equation \ref{SheetsonCpt}.
\end{proof}

%%%%%%%%%%%%%%%Margulis function %%%%%%%%%%%%%%%%%%%%%%%%%%%%
\begin{defn}
    Let $\delta\in(0,1)$ and define $f_Y:Y\to [0,\infty)$

$$f_Y(y)=\begin{cases}\sum_{v\in I_Y(y)}\|v\|^{-\delta}&\text{if }I_Y(y)\ne\emptyset\\
\inj(y)^{-\delta} &\text{otherwise}
\end{cases}$$
\end{defn}

We will need the following log continuity of $f_Y$ on the compact part. See also Proposition 36 of \cite{SS22}.

\begin{prop}\label{rmk:logctyofMF}
    The function $f_Y$ is log continuous on the compact part $X_\eta$. That is, for any $K\subset G$ and $y\in Y\cap X_\eta$, there exists $\sigma_f=\sigma_f(K,X_\eta)\ge1$ such that
    \begin{equation}\label{eqn:logctyofMF}
        \sigma_f ^{-1}f_Y(y)\le f(ky)\le \sigma_f f_Y(y)
    \end{equation}

\end{prop}

\begin{proof}

 Since $K\subset G$ is compact, there exists $R_K\ge1$ so that 
 $$R_K ^{-1}\|v\|\le\|\Ad(k)v\|\le R_K \|v\|$$
 for every $k\in K$ and $v\in Ge_1\setminus\{0\}$.
Let $\delta\in(1/2,1)$ and $y\in Y$. 

Let $\varepsilon= R_K ^{-1}\varepsilon_0\inj(y)$.

Note that if $I_Y(ky)$ is empty, then, $f_Y(k y)=\inj(k y)^{-\delta}$.

If $I_Y$ is not-empty, then for any $k\in K$,
$$f_Y(k y) = \sum_{v\in I_Y(k y),\|v\|<\varepsilon}\|v\|^{-\delta} + \sum_{v\in I_Y(k y),\|v\|\ge \varepsilon}\|v\|^{-\delta}.$$
Focusing on the second term, we have by Lemma \ref{Lem:SheetsBound},
$$\sum_{v\in I_Y(k y),\|v\|\ge \varepsilon}\|v\|^{-\delta}\le \ref{C:sheets}R_K^\delta\varepsilon_0 ^{-\delta}\inj(y)^{-\delta}\vol(Y) \ll_K \inj(k y)^{-\delta}\vol(Y)$$
where we have used the log continuity of $\inj$ in the last line and ``$\ll_K$" indicates an implicit constant that depends on $K$. 

Focusing on first term of the sum, we note that for any $v\in I_Y(k y)$ with $\|v\|<\varepsilon$ we have $\|\Ad(k) v\|<R_K\varepsilon=\varepsilon_0\inj(y)$ and so $k^{-1}v\in I_Y(y)$. Thus, setting $v'=k^{-1}v$, we have
$$\sum_{v\in I_Y(k y),\|v\|<\varepsilon}\|v\|^{-\delta} \le \sum_{v'\in I_Y( y)}\|k v'\|^{-\delta}.$$

    We now have a bound of the form
    $$    f_Y(k y)\le \sum_{v'\in I_Y( y)}\|\Ad(k) v'\|^{-\delta} + O_K(\inj(ky)^{-\delta}\max\{\vol(Y),1\})$$
    where the implicit constant only depends on $K$ and $\ref{C:sheets}$. In fact, there is a uniform lower bound on how small the volume of a closed orbit can be (Lemma 37 of \cite{SS22}) and so we may assume the second term is of the form $O_K(\inj(k y)^{-\delta}\vol(Y))$.
    We obtain
    \begin{equation}\label{eqn:bound on MF before avg}
        f_Y(k y)\le \sum_{v'\in I_Y( y)}\|\Ad(k) v'\|^{-\delta} + O_K(\inj(ky)^{-\delta}\vol(Y))
    \end{equation}    
    This, combined with $f_Y\ge1$, bound on $\|\Ad(k)v\|$, and log-continuity of $\inj$ (see Remark \ref{Rem:LogCty}) yield
    $$ f_Y(k y)\le \left(R_K ^\delta  + O_K(\inj(y)^{-\delta}\vol(Y)\right)f_Y(y)$$
    where we have absorbed the constant from log-continuity of $\inj$ into $O_K$.

    By using that $y\in X_\eta$, the proof is finished for the upper bound. The lower bound is similar.
\end{proof}

We will use a bound on the integral of $f_Y$ to obtain coarse dimension one in the next section. While the constant appearing in the following theorem depends on $\delta$, we remark that we will fix a $\delta$ close to 1 in the proof of the main results.

%%%%%%%%%%%%%%%%%%%%%%%%%%%%%%%%%
\begin{thm}\label{thm:VolBound}
    There exists a  constant $\constE\label{C:integrable}=\ref{C:integrable}(\delta,\ref{C:sheets})>0$ such that
    $$\int_Y f_Y(y)\,d\mu_Y(y)\le \ref{C:integrable}\vol(Y).$$  
\end{thm}

\begin{proof}

We first begin by showing $f_Y$ satisfies a Margulis inequality (see Equation~\eqref{eqn:iterateMI2}). 

Let $K$ denote the compact subset $ \{(a_{m_\delta}u_\s) ^{\pm1}:\s\in B_1 ^U\}$ and note there exists some $R=R(\delta)\ge 1$ with $\|\Ad(k)v\|\le R_\delta\|v\|$ for any $v\in Ge_1\setminus\{0\}$ and $k\in K$ . We note that $R_\delta\to\infty$ as $\delta\to1$. 

Let $\delta\in(1/2,1)$ and $y\in Y$. 

Then, letting $k=a_{m_\delta}u_\s$ and using Equation \ref{eqn:bound on MF before avg} we obtain
\begin{equation}\label{eqn preint}
    f_Y(a_{m_\delta}u_\s y)\le \sum_{v'\in I_Y( y)}\|\Ad(a_{m_\delta}u_\s) v'\|^{-\delta} + O_\delta(\inj(a_{m_\delta}u_\s y)^{-\delta}\vol(Y))
\end{equation}        
Denote by $A$ the integral operator defined by $$A(\varphi)(y)=\int_{B_1 ^U}\varphi(a_{m_\delta}u_\s y)d\s.$$
Integrating Equation \ref{eqn preint} over $B_1 ^U$ and applying Corollary \ref{Cor:LA}, we obtain
\begin{equation}\label{Eqn:MargulisInequality}
    A( f_Y( y))\le e^{-1}f_Y(y)+O(A(\inj(y)^{-\delta})\vol(Y))
\end{equation}
  and the implicit constant depends only on $K=K(\delta)$ and $\ref{C:sheets}$. By iterating this operator we deduce
 \begin{equation}\label{eqn:iterateMI}
     A^{(n)}(f_Y)(y)\le e^{-n}f_Y(y)+E_1\vol(Y)\sum_{j=1}^n e^{j-n}A^{(j)}(\inj(y)^{-\delta})
 \end{equation}
 where $E_1$ depends only on $K=K(\delta)$ and $\ref{C:sheets}$.  By iterating Corollary \ref{Cor:LAInj}, we have $$A^{(n)}(\inj^{-\delta}(y))\le e^{-n}\inj^{-\delta}(y)+B$$
  where $B$ depends only on $X$. Combining this with Equation~\eqref{eqn:iterateMI}, we deduce
   \begin{equation}\label{eqn:iterateMI2}
     A^{(n)}(f_Y)(y)< e^{-n}f_Y(y)+E_1\vol(Y)n e^{-n}\inj(y)^{-\delta} + 2BE_1\vol(Y)
 \end{equation}

That inequalities such as the above imply integrability is now standard, see e.g.  Lemma 11.1 of \cite{MR3418528}, Proposition 40 of \cite{SS22}, or Lemma 7.3 of Mohammadi-Oh  \cite{MR4556221}. We include a short proof for completeness. 
  
 Equation \eqref{eqn:iterateMI2} implies, $$\limsup_nA^{(n)}(f_Y)(y)\le 1+2E_1\vol(Y)B.$$
 Note that the operator norm of $A$ is bounded. This together with the fact that $(H,\mu_Y)$ is mixing implies that $\mu_Y$ is ergodic with respect to the operator $A$. Thus, by the Chacon-Ornstien theorem, for every $\varphi\in L^1(Y,\mu_Y)$ and $\mu_Y$-a.e. $y\in Y$, we have $\frac{1}{N+1}\sum_{n=0}^NA^{(n)}(\varphi)(y)\to\int \varphi d\mu_Y$. 

 For every $k\in\mathbb N$, let $\varphi_k:=\min\{k,f_Y\}$. There exists a full measure set $Y_0$ so that for every $y\in Y_0$ and every $k$, there exists some $N_{k,y}$ so that if $N\ge N_{n,y}$, then $\frac{1}{N+1}\sum_{n=0}^NA^{(n)}(\varphi_k)(y)\ge.5\int\varphi_kd\mu_Y$.

 Let $y\in Y_0$. Then the above estimate and the bound on the $\limsup A^{(n)}$ imply that $
 \int \varphi_k d\mu_Y\le 2(1+2E_1\vol(Y)B)$ for all $k$. Hence, by the monotone convergence theorem, we conclude $$\int f_Yd\mu_Y\le 2(1+2E_1\vol(Y)B).$$
 This implies a bound of the form $\ref{C:integrable}\vol(Y)$ for some $\ref{C:integrable}=\ref{C:integrable}(\delta,\ref{C:sheets}).$

\end{proof}

%%%%%%%%%%%%%%%%%%%%%%%%%%%%%%%%%%%%%
    %%%%%%%%%%%%%%%%%%%%%%%%%%%%%%%%%%%%%%%%%%%%%%%%%%%%%%%%%%%%%%%%%%%%%%%%%%%%%%%%%%%%%%%%%%%%%%%%%%%%%%%%%%%%%%%%%%%%%%%%%%%%%%%%%%%%%%%%%%%%%%%%%%%%%%%%%%%%%%%%%%%%%%%%%%%%%%%%%%%%%%%%%%%%%%%%%%%%%%%%%%%%%%%%%%%%%%%%%%%%%%%%%%%%%%%%%%%%%%%%%%%
%

\section{ Coarse dimension, input lemma, and a recentering Lemma}\label{sec:input}
   Recall that $1/2<\delta<1$ and so it can be arbitrarily close to 1. We will use Theorem \ref{thm:VolBound} to show that there is a subset of $B_\mathfrak r(0,1)$ that has coarse dimension of at least $\delta$, that is, at least 1. This is the second item in the following proposition.

\begin{prop}\label{prop:input}
Let $1/2<\delta<1$. There exists $y_0\in Y$ and a finite subset $F\subset B_{\mathfrak{r}}(0,1),$ containing 0 such that $\#F\gg \vol(Y)$  and so that
\begin{enumerate}
    \item ${\exp(w)y_0}\in Y$ for all $w\in F$
    \item $\sum_{v\in F\setminus\{w\}}\|w-v\|^{-\delta}\ll \#F$ for all $w\in F$.
\end{enumerate}    
The implicit constants depend on $\eta_X$ and the dimension of the group.
\end{prop}
\begin{proof}
   Let $\eta=\frac{1}{10}\eta_X$ and $\beta = \frac{\eta^2}{\dim(G)\ell}$ where $l\ge2$ is to be determined. Essentially, $\ell$ is large enough so that $\beta<\beta_0$ so that Lemma \ref{lem: normal form} applies and as such depends only on the dimension of $G$.

Non-divergence (quantitative non-divergence result of unipotent flows of Dani \cite{MR530631} ) allows us to obtain $\mu_Y(X_{\eta_X})\ge.9$, by potentially choosing smaller $\eta_X$. Since  $X_{\eta_X}\subset X_{2\eta}$, then we get $\mu_Y(X_{2\eta})\ge.9$.

Let $Y'=\{y\in Y\cap X_{2\eta}:f_Y(y)\le 100\ref{C:integrable}\vol(Y)\}$.
Note, by Chebyshev and Theorem \ref{thm:VolBound}, $$\mu_Y\{y\in X_{2\eta}:f_Y(y)>100\ref{C:integrable}\vol(Y)\}\le \frac{1}{100\ref{C:integrable}\vol(Y)}\int f_Y \,d\mu_Y\le\frac{1}{100}.$$
Combining this with the measure of $X_{2\eta}$, we get $\mu_Y(Y')\ge.88$.

There exists some $z\in X_{2\eta}$ and constant $c>0$ such that $
\mu_Y(B^G_{\beta^2}z\cap Y')\ge c$.  Indeed,  if not, then since $\{B^G_{\beta^2}z:z\in X_{2\eta}\}$ is an open cover of the compact set $X_{2\eta}$ and each set in the cover has $\mu_Y$-measure $\frac{ \beta^{2\dim(H)}}{\vol(Y)}$, by choosing a finite subcover $\{B^G _{\beta^2}z_i:i\in\mathcal I\}$, we see that $\frac{.88}{|\mathcal I|}<c$. So picking $c = \frac{.88}{2|\mathcal I|}$ gives a contradiction.

Let $y_0\in B^G_{\beta^2}z\cap Y'$ and $Y''=\{y\in Y\cap X_{2\eta}:f_Y(y)\le 100\sigma_f\ref{C:integrable}\vol(Y)\}$
where $\sigma_f$ is from Proposition \ref{rmk:logctyofMF}.

Define
$$F:=\{w\in I_Y(y_0):B^G _{\beta^2}z\cap B^H _\beta\exp(w)y_0\cap Y''\ne\emptyset\}.$$
We remark that by using that since we are working in $B_{\beta^2} ^Gz$ and $\beta^2\ll\eta_X^2$, that each $w\in F$ has size $\ll \beta^2<\beta$ for large enough $\ell$.

Now we claim that 
\begin{equation}\label{eqn: goodset}
    B^G_{\beta^2}z\cap Y'\subseteq \bigcup_{w\in F}B^H_\beta\exp(w)y_0
\end{equation} 
Indeed, suppose $y\in B^G_{\beta^2}z\cap Y'$ and so it has the form \ $y=h\exp(w)z$ where $h\in B_{\beta^2} ^H$ and $\|w\|\le \beta^2$. Similarly, we can write $y_0\in B^G_{\beta^2}z\cap Y'$ as $y_0=h_0\exp(w_0)z$ where $h_0\in B_{\beta^2} ^H$ and $\|w_0\|\le \beta^2$. Solving this for $z$ and plugging into the expression for $y$ yields
\begin{align*}
y&=h\exp(w)\exp(-w_0)h_0^{-1}y_0 = hh_0^{-1}\exp(\Ad(h_0)w)\exp(-\Ad(h_0)w_0)y_0\\
&= hh_0^{-1}\exp(w_1)\exp(-w_2)y_0
\end{align*}
where $w_1:=\Ad(h_0)w$ and $w_2:=\Ad(h_0)w_0$ satisfy that their norm is $\ll\dim(G)\beta^2.$
Recall that $\beta = \frac{\eta^2}{\dim(G)\ell}$. By choosing  $\ell$ large enough, we can guarantee that $\beta^2\le\beta_0$ and so by Lemma \ref{lem: normal form}, we have 
$$\exp(w_1)\exp(-w_2)=h'\exp(w')$$
with $w'\in \mathfrak{r}$ and $h'=\exp(v_\mathfrak{h})\in H$ satisfy $\|w'\|\ll 4\beta^2$ and   $\|v_\mathfrak{h}\|\ll4\|w'\|\beta^2.$ By potentially further making $\ell$ larger, we have that $\|w'\|<\varepsilon_0\inj(y_0)$ (See Definition \ref{def: sheets}). By log continuity of $f_Y$ (Proposition \ref{rmk:logctyofMF}) and potentially shrinking $\beta$, we can further guarantee that $w\in F$. As such, we have 
$$y=hh_0 ^{-1}h'\exp(w')y_0$$
where $w'\in F$ and $h,h_0\in B_{\beta^2} ^H$ and $h'=\exp(v_\mathfrak{h})$ where $\|v_\mathfrak{h}\|\ll16\beta^4$. In particular, we have $hh_0^{-1}h'\in B^H_\beta$

Using Equation~\ref{eqn: goodset} and that each sheet contributes $\frac{ \beta^{\dim(H)}}{\vol(Y)}$, then we have 
$$c\beta ^{-\dim(H)}\vol(Y)\le \#F.$$

Thus, we have found a finite subset $F\subseteq B_{\mathfrak{r}}(0,1)$ with $\#F\gg \vol(Y)$ that addresses item (1) of the proposition.

We now turn to item (2) of the proposition.

We show that for every $w\in F$, we have $\sum _{v\in F-\{w\} }\|w-v\|^{-\delta}\ll \#F$. We will do this by comparing the sum to $f_Y(\exp(w)y_0)$. Additionally, we will show that the value of the later is comparable to the volume $\vol(Y)$, then item (1) shows that $\sum _{v\in F-\{w\} }\|w-v\|^{-\delta}\ll \#F$.

Let $w_0\in F$ and $w\in F\setminus \{w_0\}$. Then,
\begin{align*}
    \exp(w)y_0 &= \exp(w)\exp(-w_0)\exp(w_0)y_0\\
    &= h_w\exp(v_w)\exp(w_0)y_0\qquad \text{(Lemma \ref{lem: normal form})}
\end{align*}
 and for some $h_w\in H$, $v_w\in \mathfrak{r}$, and $\|h_w-I\|\ll \beta\|v_w\|$. 

 Since $w_0,w\in B_\mathfrak{r}(0,\beta)$, then $\|h_w-I\|\ll \beta^2$ so that by choosing $\beta$ small enough (i.e. $l$ large enough) we have that $h_w ^{\pm}\in B_\beta ^H$. Thus,
$$\exp(v_w)\exp(w_0)y_0=h_w^{-1}\exp(w_0)y_0\in Y.$$

Additionally, by Lemma \ref{lem: normal form}, we have $\|v_w\|\le2\|w_0-w\|$. Thus, we have $\|v_w\|\le 16\beta^2\le \varepsilon_0\inj(\exp(w_0)y_0)$ for sufficiently  large $l$ and so $v_w\in I_Y (\exp(w)y_0).$ (See Definition \ref{def: sheets}).

Since $\exp(w)y_0\ne \exp(w')y_0$, for $w\ne w'\in F\subset I_Y(y_0)$, then the map $w\mapsto v_w$ is well defined and injective.

We conclude,
$$\sum_{w\in F\setminus \{w_0\}}\|w-w_0\|^{-\delta}\le 2^\delta \sum_{w\in F\setminus \{w_0\}}\|v_w\|^{-\delta}\le 2^\delta \sum_{v\in I_Y(\exp(w)y_0) }\|v\|^{-\delta}=2 f_Y(\exp(w)y_0)$$
where on the last equality we used $\delta\le1$ and the definition of $f_Y$.

By the definition of $F$, we have that since $w\in F$, that there is some $h'\in B^H_\beta$ with $h\exp(w)y_0\in Y'$ and so $f_Y(h\exp(w)y_0)$ is comparable to the volume. By log continuity of $f_Y$ (Proposition \ref{rmk:logctyofMF}), we have $f_Y(\exp(w)y_0)\ll f_Y(h'\exp(w)y_0)\ll \vol(Y)$.

Hence, $\sum_{w\in F\setminus \{w\}}\|w-v\|^{-\delta}\ll \vol(Y)$. By item (1), since $\vol(Y)\ll\#F$, then we have satisfied item (2).
\end{proof}

In this next sections we will use a projection theorem to project the dimension gained in Proposition \ref{prop:input} into the expanding direction of the horospherical subgroup. Notably, this is the only direction of the horospherical subgroup that is not already contained in $H$. However, it may occur that the elements we are working with have a too large of a size  and the error cannot be controlled in the argument we wish to give to prove Theorem \ref{Thm:EffDen}. As such, we need the following lemma that allows us to replace the sets with subsets where errors can be controlled. This lemma is essentially from \cite{BFLM}, but is reproved in Appendix C of Lindenstrauss--Mohammadi \cite{MR4549089} to explicate certain constants. The proof their works in higher dimensions and so we only state the lemma. See also Section 2 of Bourgain \cite{MR2763000}. Recall that the dimension of $\mathfrak r$ is $n$.

\begin{lem}\label{lem:Localize} (Localizing Lemma)

    Let $F\subset B_\mathfrak r (0,1)$ be a subset that satisfies 
    $$ \sum_{w'\in F\setminus\{w\}}\|w-w'\|^{-\delta} \le D \cdot(\#F)^{1+\varepsilon}.$$  For every $0<\varepsilon\ll_{\dim(G)}\delta$, there there exists $w_0\in F$, $b_1>0$, with
$$(\#F)^{-\frac{n-\delta+(2+n)\varepsilon}{n-\delta +(3+n)\varepsilon}}\le b_1\le(\#F)^{-\varepsilon},$$
and a subset $F'\subset B_\mathfrak r (w_0,b_1)\cap F$ so that the following holds. Let $w\in \mathfrak r$ and $b\ge (\#F)^{-1}$. Then
$$\frac{\#(F'\cap B(w,b))}{\#F'}\le C' \cdot\left(\frac{b}{b_1}\right)^{\delta -(n+3)\varepsilon}$$
where $C'\ll_D \varepsilon^\star$ with absolute implied constants.

\end{lem}

%%%%%%%%%%%%%%%%%%%%%%%%%%%%%%%%%%%%%%%%%%%%%%%%%%%%%%%%%%%%

%%%%%%%%%%%%%%%%%%%%%%%%%%%%%%%%%%%%%%%%%%%%%%%%%%%%%%%%%%%%%%%%%%%%%%%%%%%%%%%%%%%%%%%%%%%%%%%%%%%%%%%%%%%%%%%%%%%%%%%%%%%%%%%%%%%%%%55

\section{Projection Theorem}\label{sec:proj}
In this section we prove Theorem \ref{thm:ProjSO(n,1)} which will allow us to project the dimension gained in Proposition \ref{prop:input} into the expanding direction of the horospherical subgroup. In fact, we prove a more general restricted projection theorem of $n-2$ dimensional families of projections from $\mathbb R^n$ to $\mathbb R$ and apply it to the adjoint representation of the maximal unipotent subgroup of $\SO(n-1,1)$ on the Lie algebra of $\SO(n,1)$ to deduce Theorem \ref{thm:ProjSO(n,1)}.

Let us begin by fixing some notation. Recall that we use the following coordinates for $\R^n$,
\[
\R^{n}=\{(r_1,\c, r_2): r_i\in \R, \c\in\R^{n-2}\}. 
\]
Let $L:\R^{n-2}\to\R^{n-2}$ be an invertible linear map, and let $q:\R^{n-2}\to\R$ be a positive definite quadratic form. For every $\s\in\R^{n-2}$, define $\pi_\s=\pi_{L, q, \s}:\R^{n}\to\R$ by 
\[
\pi_{\s}(r_1, \c, r_2)= r_1 + \c\cdot L(\s) + r_2q(\s)
\]
where $\c\cdot L(\s)$ is the usual inner product on $\R^{n-2}$. 

The following theorem is the main step in the proof of Theorem~\ref{thm:ProjSO(n,1)}.

\begin{thm}\label{thm: main finitary}
    Let $0<\delta\leq 1$, and let  $0< b_0\leq1$. 
Let $F\subset B_{\R^n}(0,1)$ be a finite set satisfying the following: 
\[
\#(B_{\R^n}(Z,  b)\cap F)\leq C\cdot  b^\delta\cdot (\# F)\quad\text{for all $Z\in F$ and all $ b\geq  b_0$}
\]
where $C\geq 1$.

Let $0<\varepsilon<\delta/100$.
For every $ b\geq  b_0$, there is a subset 
$B_b\subset B_{\R^{n-2}}(0,1):=\{\s\in\R^{n-2}: \norm{\s}\leq 1\}$ 
with 
\[
|B_{\R^{n-2}}(0,1)\setminus B_{ b}|\leq \hat C \varepsilon^{-A} b^{\varepsilon/10}
\]
so that the following holds. 
Let ${\bf s}\in B_{ b}$, then there exists $F_{ b,{\bf s}}\subset F$ with 
\[
\#(F\setminus F_{ b,{\bf s}})\leq \hat C \varepsilon^{-A} b^{\varepsilon/10}\cdot (\# F)
\]
such that for all $Z\in F_{ b,{\bf s}}$, we have 
\[
\#\{Z'\in F_{ b, \s}: |\pi_{\bf s}(Z')-\pi_{\bf s}(Z)|\leq  b\}\leq C\hat C b_0^{-10\varepsilon}\cdot  b^{\delta}\cdot (\# F),
\] 
where $A$ is absolute and $\hat C$ depends on $L$ and $q$. 
\end{thm}

%%%%%%%%%%%%%%%%%%%%%%%%%%%%%%%%%%%%%%

 The most difficult case of Theorem~\ref{thm: main finitary}  is arguably the case $n=3$, which was studied in~\cite{kenmki2017marstrandtype, PYZ} using fundamental works of Wolff and Schlag~\cite{Wolff, Schlag} --- see also~\cite{GGW} for a different approach to a more general problem in the same vein. Indeed, when $n>3$, it is plausible that one may deduce the Theorem~\ref{thm: main finitary} using techniques of~\cite{SchPr} more directly; we, however, take a slightly different route which is a hybrid of these two methods.  

 While the main application of Theorem~\ref{thm: main finitary} is to project the dimension into the expanding direction of the horospherical subgroup, we motivate Theorem~\ref{thm: main finitary} by stating the qualitative version. Strictly speaking, Theorem~\ref{thm: MarstrandProj} is not used in the sequel; we include the theorem and its short proof for completeness, and in the hope that the reader may find it illuminating for the proof of the finitary version, Theorem~\ref{thm: main finitary}. 

\begin{thm}\label{thm: MarstrandProj}
    Let $\cpct\subset \R^n$ be a compact subset. Then for Lebesgue almost every $\s\in \R^{n-2}$, we have 
    $$
    \hdim(\pi_\s(\cpct))=\min(1,\hdim \cpct)
    $$
    where $\hdim\!$ denotes the Hausdorff dimension.
\end{thm}

We conclude with a proof of Theorem \ref{thm:ProjSO(n,1)} assuming Theorem~\ref{thm: main finitary}. We then prove Theorem~\ref{thm: main finitary} and Theorem~\ref{thm: MarstrandProj} in the next subsections. 

%%%%%%%%%%%%%%%%%%%%
\begin{proof}[Proof of Theorem \ref{thm:ProjSO(n,1)}]\label{sec:proof proj}
First note that replacing $F$ by $\frac{1}{b_1}F$, we may assume $b_1=1$.  

Recall from~\S\ref{Preliminaries} that we may identify $\mathfrak r$ with 
\[
\R^{n}=\{(r_1,\c, r_2): r_i\in \R, \c\in\R^{n-2}\},
\]
and that $U=\{u_{\bf s}: \s\in\R^{n-2}\}$. Moreover, by~\eqref{eqn:xi_s}, if we let $Z=Z(r_1,\c,r_2)$, then
\begin{equation}\label{eqn:xi_s'}
\xi_{\s}(Z)=\left(\Ad(u_\s)Z(r_1,\c,r_2)\right)^+:=(r_1+\s\cdot \c+r_2\|\s\|^2/2)
\end{equation}
Therefore, Theorem \ref{thm: main finitary} is applicable with $L(\s)=\s$ and $q(\s)=\frac{1}{2}\|s\|^2$. 

We may assume $b$'s are dyadic numbers, in particular, we put $b_0=2^{-k_0}$. 
Fix some $0<\varepsilon<\delta/100$, and let $k_1$ be so that 
\begin{equation}\label{eqn:choose ell1}
\hat C\varepsilon^{-A} \sum_{k=k_1}^{\infty} 2^{-\varepsilon k/10}< 0.01\bigl|B_1^U\bigr|
\end{equation}

Apply Theorem \ref{thm: main finitary} with $b=2^{k}$ for $k_1\leq k\leq k_0$  
Let $B=\bigcap_{k=k_1}^{k_0} B_{2^{-k}}$. 
Then Theorem~\ref{thm: main finitary} and~\eqref{eqn:choose ell1} imply that $|B_1^U\setminus B|\leq 0.01|B_1^U|$.

For every $u_{\bf s}\in B$, let $F_{\bf s}=\bigcap_{k=k_1}^{k_0}F_{2^{-k},{\bf s}}$.
Then by Theorem~\ref{thm: main finitary} and and~\eqref{eqn:choose ell1}, we have $\#(F\setminus F_{\bf s})\leq 0.01\cdot (\#F)$.
Moreover, for all $Z\in F_{\bf s}$ and all $k_1\leq k\leq k_0$
we have
\[
\#\{Z'\in F_{\bf s}: |\xi_{\bf s}(Z')-\xi_{\bf s}(Z)|\leq 2^{-\ell}\})\leq \hat C C  2^{-\varepsilon k_0}\cdot  2^{\delta k}\cdot (\# F).
\]

This implies Theorem \ref{thm:ProjSO(n,1)} with $B$, $F_{\bf s}$, and $C_{\varepsilon}=\hat C C 2^{1+k_2}$ --- note that since $0\leq\delta\leq 1$ and we have included the factor $2^{1+k_2}$ in the definition of $C_{\varepsilon}$, the above holds for all scales $b\geq 2^{-k_0}$ as claimed in the theorem. 
\end{proof} 
%%%%%%%%%%%%%%%%%%%%%%%%%%%%
\subsection{Proof of  Theorem~\ref{thm: MarstrandProj}}

In this section we prove Theorem~\ref{thm: MarstrandProj} to help motivate the statement and proof of Theorem~\ref{thm: main finitary}.

%It will be more convenient to work with $B:=\{{\bf s}\in\mathbb R^{n-2}: 1/2\leq \|{\bf s}\|\leq 1\}$. 
%The statement in Theorem~\ref{thm: main finitary} regarding $\{{\bf s}\in\mathbb R^{n-2}: \|{\bf s}\|\leq 1\}$ 
%follows from this as the estimate

We begin with an elementary lemma. For every $\s\in\R^{n-2}$, define 
\begin{equation}\label{eqn: def fs}
f_\s=f_{L,q,\s}:\R^n\to\R^3\quad\text{by}\quad f_\s(r_1,\c, r_2)= \Bigl(r_1, \c \cdot L(\s), r_2q(\s) \Bigr),
\end{equation}
where $L$ is a bijection of $\mathbb R^{n-2}$ and $q$ is a positive definite quadratic form on $\mathbb R^{n-2}$.

\begin{lem}\label{lem: transversality of ft}
There exists $\eta_0=\eta_0(q, L)\leq 0.01$ so that the following holds. 
Let $\ell\in\mathbb N$ and let $0<\eta<2^{-2\ell}\eta_0$. 
Then for all $Z, Z'\in\R^n$ satisfying $\norm{Z-Z'}=1$, we have   
\[
\absolute{\{\s\in \mathsf {Cyl}_\ell: \norm{f_\s(Z)-f_\s(Z')}\leq \eta\}}\ll \eta \absolute{\mathsf {Cyl}_\ell}
\]
where $\mathsf {Cyl}_\ell=\{\s\in\R^{n-2}: 2^{-\ell}\leq \norm{\s}\leq 2^{-\ell+1}\}$ and the implied constant depends on $L$ and $q$.
\end{lem}

\begin{proof}
    Let us write $Z=(r_1,\c, r_2)$ and $Z'=(r'_1,\c', r_2')$. Let us denote the set in the lemma by $S$. 
    If $S\neq \emptyset$, then there exists some $\s\in \mathsf {Cyl}_\ell$ so that 
    \[
    \norm{\Bigl(r_1-r'_1, (\c-\c')\cdot L(\s), (r_2-r'_2)q(\s)\Bigr)}\leq \eta.
    \]
    Thus $\absolute{r_1-r_1'}, \absolute{r_2-r_2'}q(\s)\leq \eta$. 
    Since $\s\in\mathsf {Cyl}_\ell$, we have $q(\s)\gg 2^{-2\ell}$, $\min_{\norm{\s}=1} \absolute{q(\s)}$. 
    These and $\norm{Z-Z'}=1$, thus imply
    \[
    \norm{\c-\c'}\gg 1.
    \]
        
    Altogether, either $S=\emptyset$ in which case the proof is complete, or we may assume $\norm{\c-\c'}\gg1$ and
    \[
    S\subset \{\s\in \mathsf {Cyl}_\ell: \norm{(\c-\c')\cdot L(\s)}\leq \eta\}.
    \]
    Since $\norm{\c-\c'}\gg 1$ and $\eta<2^{-2\ell}\eta_0$, the measure of the set on the right side of the above is $\ll\eta\absolute{\mathsf {Cyl}_\ell}$, where the implied constant depends on $L$ and $q$. 
    
    The proof is complete. 
\end{proof}

Recall, from Kaufman \cite{MR248779}, that the $\delta$-dimensional energy of a probability measure $\nu$ on $\R^d$ is defined by 
\[
\eng_\delta(\nu)=\iint\frac{\diff\!\nu(Z)\diff\!\nu(Z')}{\|Z-Z'\|^\delta}
\]
We also recall that that the supremum over all Borel probability measures $\nu$ on a set $A$ of the above quantity returns that $\delta$-capacity of $A$. Furthermore, the supremum over positive $\delta$-capacities yields the Hausdorff dimension of $A$.

\begin{lem}\label{lem: ft dim preserving}
 Let $0<\delta<1$. Let $\nu$ be a probability measure supported on $B_{\R^n}(0,1)$ which satisfies 
 \[
 \eng_\delta(\nu)\leq C
 \]
 for some $C\geq 1$. Then for every $\ell\in\mathbb N$ and all $R>0$, we have 
 \[
 \absolute{\{\s\in\mathsf{Cyl}_\ell: \eng_\delta(f_\s\mu)>R\}}\leq C'/R
 \]
 where $C'\ll \frac{2^{2\ell}C}{1-\delta}$. In particular, $\eng_\delta(f_\s\nu)<\infty$ for (Lebesgue) a.e.\ $\s$. 
\end{lem}

\begin{proof}
We recall the standard argument which is based on Lemma~\ref{lem: transversality of ft}. 
Since $\ell$ is fixed, we write $\mathsf{Cyl}$ for $\mathsf{Cyl}_\ell$.

Put $\nu_\s=f_\s\nu$. Using the definition of 
$\delta$-dimensional energy and Fubini's theorem, we have 
\begin{align*}
    \int_{\mathsf{Cyl}}\eng_\delta(\nu_\s)\diff\!\s&= \int_{\mathsf{Cyl}}\int_{\R^n}\int_{\R^n}\frac{\diff\!\nu(Z)\diff\!\nu(Z')}{\norm{f_\s(Z)-f_\s(Z')}^\delta}\diff\!\s\\
    &=\int_{\R^n}\int_{\R^n}\int_{\mathsf{Cyl}}\frac{\diff\!\s}{\norm{f_\s(Z)-f_\s(Z')}^\delta} \diff\!\nu(Z)\diff\!\nu(Z')
\end{align*}
Renormalizing with the factor $\norm{Z-Z'}^{\delta}$, we conclude that
\begin{equation}\label{eq: int energy mut}
\int_{\mathsf{Cyl}}\eng_\delta(\nu_\s)\diff\!\s=\int_{\R^n}\int_{\R^n}\int_{\mathsf{Cyl}}\frac{\diff\!\s}{\norm{\frac{f_\s(Z)-f_\s(Z')}{\norm{Z-Z'}}}^\delta}
\frac{\diff\!\nu(Z)\diff\!\nu(Z')}{\norm{Z-Z'}^\delta}.
\end{equation}
Since $\delta<1$, applying Lemma~\ref{lem: transversality of ft}, we conclude that
\[
\int_{\mathsf{Cyl}}  \norm{\frac{f_\s(Z)-f_\s(Z')}{\norm{Z-Z'}}}^{-\delta}\diff\!\s\ll \frac{2^{2\ell}}{1-\delta}.
\]
This,~\eqref{eq: int energy mut}, and $\eng_\delta(\nu)\leq C$ imply that   
\[
\int_B\eng_\delta(\nu_\s)\diff\!\s\ll2^{2\ell}\frac{\eng_\delta(\nu)}{1-\delta}\ll\frac{2^{2\ell}C}{1-\delta}.
\]
The claim in the lemma follows from this and the Chebyshev's inequality. 
\end{proof}

\begin{proof}[Proof of Theorem~\ref{thm: MarstrandProj}]
Let $s\in\R$ and $\s\in\R^{n-2}$. Then 
\begin{equation}\label{eq: decomp pi t}
\begin{aligned}
    \pi_{s\s}(r_1,\c, r_2)&=r_1+\c\cdot L(s\s)+r_2q(s\s)\\
    &=(1, s, s^2)\cdot f_\s(r_1,\c, r_2). 
\end{aligned}
\end{equation}
    
Let $\ell\in\mathbb N$ be arbitrary. Let $\cpct\subset B_{\R^n}(0,1)$ be a compact subset, and let  $\kappa=\min(1,\hdim \cpct)$. Let $0<\delta<\kappa$. 
By Frostman's lemma, there exists a probability measure $\nu$ supported on $\cpct$ and satisfying the following   
\[
\nu(B(Z, b))\leq  b^\delta \quad\text{for all $Z\in \cpct$}.
\]
Then by Lemma~\ref{lem: ft dim preserving}, applied with $\nu$, 
there exists a conull subset $\Xi_\delta\subset\mathsf{Cyl}_\ell$ so that  $\hdim(f_\s(\cpct))\geq \delta$ for all $\s \in\Xi_\delta$. Applying this with $\delta_n=\kappa-\frac{1}{n}$ for all $n\in\N$, we obtain a conull subset $\Xi\subset\mathsf{Cyl}_\ell$ so that 
\[
\hdim(f_\s(\cpct))\geq \kappa,\quad\text{for all $\s \in\Xi$.}
\]

Let $\s\in\Xi$, and set $\cpct_\s=f_\s(\cpct)$. Then by~\cite[Thm.\ 1.3]{PYZ}, see also~\cite{GGW}, for a.e.\ $s\in\R$, we have 
\[
\Bigl\{x_1+x_2s+x_3s^2: (x_1, x_2, x_3)\in \cpct_\s\Bigr\}\subset \R
\]
has dimension $\kappa$. This and~\eqref{eq: decomp pi t} complete the proof.  
\end{proof}

\subsection{Proof of Theorem~\ref{thm: main finitary}}\label{sec: proof of finitary}

In this subsection we turn to the proof of Theorem~\ref{thm: main finitary}. The proof is a discretized version of the argument we used to prove Theorem~\ref{thm: MarstrandProj}.

In what follows, it is more convenient to work with the uniform measure rather than the counting measure. 
Let  $\unm$ be the uniform measure on $F$. 
Our standing assumption is that for some $0< \delta\leq 1$ and some $C\geq 1$, we have  
\begin{equation}\label{eq: dimension alpha}
\unm(B_{\R^n}(Z,  b))\leq C b^\delta\qquad\text{for all $Z\in F$ and all $ b\geq  b_0$}.
\end{equation}

For a finitely supported probability measure $\nu$ on $\R^d$, define  
\[
\hat\eng_{\delta,\nu}:\R^d\to\R\quad\text{by}\quad \hat\eng_{\delta,\nu}(Z)=\int \norm{Z-Z'}_+^{-\delta}\diff\!\nu(Z')
\]
where $\norm{Z-Z'}_+=\max\Bigl(\norm{Z-Z'}, b_0\Bigr)$ for all $Z,Z'\in\R^d$ --- this definition is motivated by the fact that we are only concerned with scales $\geq  b_0$, 

\medskip

We recall the following standard lemma. 

\begin{lem}\label{lem: truncated energy}
    Let $\nu$ be a finitely supported probability measure on $\R^d$. Assume that for some $Z\in \R^d$ we have 
    \[
    \hat\eng_{\delta,\nu}(Z)\leq R.
    \]
    Then for all $ b\geq b_0$, we have 
    \[
    \nu(B_{\R^d}(Z, b))\leq R b^{\delta}.  
    \]
\end{lem}

\begin{proof}
We include the proof for completeness. 
Let $ b\geq  b_0$, then  
\begin{align*}
     b^{-\delta}\nu(B_{\R^d}(Z, b))&\leq \int_{B_{\R^d}(Z, b)}\norm{Z-Z'}_+^{-\delta}\diff\!\nu(Z')\\
    &\leq \int \norm{Z-Z'}_+^{-\delta}\diff\!\nu(Z')=\hat\eng_{\delta,\nu}(Z)\leq R,
\end{align*}
as it was claimed. 
\end{proof}

Fix some $0<\varepsilon<\delta/100$ for the rest of the argument. 
Recall from Theorem~\ref{thm: main finitary}, that for every $b\geq b_0$ we seek $B_b\subset B_{\mathbb R^{n-2}}(0,1)$ with $|B_{\mathbb R^{n-2}}(0,1)\setminus B_b|\ll\varepsilon^{-A} b^{\varepsilon/10}$. Thus we may replace $B_{\mathbb R^{n-2}}(0,1)$, with 
\begin{equation}\label{eq: define B}
B:=\{\s\in\R^{n-2}: b_0^{\varepsilon/8}\leq \|s\|\leq 1\}. 
\end{equation}

Moreover, since our estimates are allowed to depend on $\varepsilon$ polynomially, we will assume $b_0$ is small enough compare to $\varepsilon$.

For every $\s\in B_{\mathbb R^{n-2}}(0,1)$, let $\unm_\s=f_\s\unm$ 
where $f_\s:\R^n\to\R$ is as in~\eqref{eqn: def fs}. That is,
\begin{equation*}
f_\s(r_1,\c, r_2)= \Bigl(r_1, \c \cdot L(\s), r_2q(\s) \Bigr),
\end{equation*}
where $L$ and $q$ are as in the statement of the theorem, and we use the coordinates $\R^{n}=\{(r_1,\c, r_2): r_i\in\R, \c\in\R^{n-2}\}$.  

\begin{lem}\label{lem: ave trun egy proj}
    For every $Z\in F$,  we have 
    \[
    \int_B\hat\eng_{\delta,\unm_\s}(f_\s Z)\diff\!\s\ll Cb_0^{\varepsilon/2}
    \]
    where the implied constant depends on $L$ and $q$. 
\end{lem}

\begin{proof}
    Let $Z\in F$. By the definitions, we have 
    \[
    \int_B\hat\eng_{\delta,\unm_\s}(f_\s Z)\diff\!\s= \int_B\int \norm{f_\s Z-f_\s Z'}_+^{-\delta}\diff\!\unm(Z')\diff\!\s.
    \]
   Renormalizing with $\norm{Z-Z'}_+^{-\delta}$ and using Fubini's theorem, we have 
   \[
   \int_B\hat\eng_{\delta,\unm_\s}(f_\s Z)\diff\!\s=\int \int_B\frac{1}{\frac{\norm{f_\s Z-f_\s Z'}_+^\delta}{\norm{Z-Z'}_+^\delta}}\diff\!\s \norm{Z-Z'}_+^{-\delta}\diff\!\unm(Z').
   \]
  Without loss of generality, we will assume $b_0=2^{-k_0}$ for some $k_0\in\N$. For every $0\leq k\leq k_0-1$, let 
 \[
 F_k(Z)=\{Z'\in F: 2^{-k-1}< \norm{Z-Z'}\leq 2^{-k}\},
 \]
 and let $F_{k_0}(Z)=\{Z'\in F: \norm{Z-Z'}\leq 2^{-k_0}\}$. 
 
 For all $Z'\in F_{k_0}$, we have $1\leq \frac{\norm{f_\s Z-f_\s Z'}_+}{\norm{Z-Z'}_+}\ll_{L,q} 1$. Thus, 
   \begin{equation}\label{eq: contribution of Fk0}
   \int_{F_{k_0}}\int_B\frac{1}{\frac{\norm{f_\s Z-f_\s Z'}_+^\delta}{\norm{Z-Z'}_+^\delta}}\diff\!\s \norm{Z-Z'}_+^{-\delta}\diff\!\unm(Z')\ll \unm(F_{k_0})2^{k_0\delta}\ll C. 
   \end{equation}
   We now turn to the contribution of $F_k$ to the above integral for $k< k_0$. 
   If $Z'\in F_k$, for some $k<k_0$, then $\norm{Z-Z'}_+=\norm{Z-Z'}$ and we have 
   \begin{multline}\label{eq: contribution of Fk}
   \int_{F_{k}}\int_B\frac{1}{\frac{\norm{f_\s Z-f_\s Z'}_+^\delta}{\norm{Z-Z'}_+^\delta}}\diff\!\s \norm{Z-Z'}_+^{-\delta}\diff\!\unm(Z')=\\ 
   \int_{F_{k}}\int_B\frac{1}{\frac{\norm{f_\s Z-f_\s Z'}_+^\delta}{\norm{Z-Z'}^\delta}}\diff\!\s \norm{Z-Z'}^{-\delta}\diff\!\unm(Z'). 
   \end{multline}
   By Lemma~\ref{lem: transversality of ft} and the definition of $B$ in~\eqref{eq: define B}, we have 
   \[
   \int_B\frac{1}{\frac{\norm{f_\s Z-f_\s Z'}_+^\delta}{\norm{Z-Z'}^\delta}}\diff\!\s \leq \int_B\frac{1}{\norm{\frac{f_\s Z-f_\s Z'}{\norm{Z-Z'}}}^\delta}\diff\!\s\ll 2^{\varepsilon k_0/4}=b_0^{-\varepsilon/4}
   \]
   where the implied constant depends on $L$ and $q$. Thus 
   \[
   \eqref{eq: contribution of Fk}\ll 2^{-\varepsilon k_0/4}\unm(F_k)2^{k\delta}\ll C2^{-\varepsilon k_0/4}
   \]
   This and~\eqref{eq: contribution of Fk0} imply that 
   \[
   \int_B\hat\eng_{\delta,\unm_\s}(f_\s(Z))\diff\!\s\ll Ck_02^{-\varepsilon k_0/4}\leq C2^{-\varepsilon k_0/2}=Cb_0^{-\varepsilon/2}
   \]
   as we claimed. 
\end{proof}

\begin{prop}\label{prop: ft preserves dim}
    There exists $B'\subset B$ with 
    \[
    \absolute{B\setminus B'}\leq b_0^{\varepsilon}
    \]
    so that the following holds. For every $\s\in B'$, there exists $F_{\s}\subset F$ with 
    \[
    \unm(F\setminus F_\s)\leq  b_0^{\varepsilon}
    \]
    so that for every $Z\in F_\s$ and every $ b\geq  b_0$ we have 
    \[
    \unm_\s\Bigl(B_{\R^3}(f_\s(Z),  b)\Bigr)\ll b_0^{-3\varepsilon}\cdot  b^{\delta}
    \]
    where the implied constant depends on $C$. 
\end{prop}

\begin{proof}
  The proof is based on Lemma~\ref{lem: ave trun egy proj} and Chebychev's inequality as we now explicate. 
  
%  First note that we may assume $ b_0$ is small enough (polynomially in $\varepsilon$) so that 
%  \[
%   b_0^{-\varepsilon/10}>\absolute{\log b_0}
%  \]
%  otherwise the statement follows trivially. 
  
  By Lemma~\ref{lem: ave trun egy proj}, we have 
  \begin{equation*}
\int_B\hat\eng_{\delta,\unm_\s}(f_\s(Z))\diff\!\s\leq C'b_0^{-\varepsilon/2}.
  \end{equation*}
  where $C'\ll_{L,q} C$. 
  
  Averaging the above with respect to $\unm$, and using Fubini's theorem we have 
  \begin{equation*}
  \int_B\int\hat\eng_{\delta,\unm_\s}(f_\s(Z))\diff\!\unm(Z)\diff\!\s\leq C'b_0^{-\varepsilon/2}.
  \end{equation*}
  Let $B'=\{\s\in B: \int\hat\eng_{\delta,\unm_\s}(f_\s(Z))\diff\!\unm(Z)<C' b_0^{-2\varepsilon}\}$.
  Then by the above equation and Chebychev's inequality we have 
  \[
  \unm(B\setminus B')\leq  b_0^{\varepsilon}.
  \] 
  Let now $\s\in B'$, then 
  \begin{equation}\label{eq: int of energy for good t}
  \int\hat\eng_{\delta,\unm_\s}(f_\s(Z))\diff\!\unm(Z)\leq C' b_0^{-2\varepsilon}.
  \end{equation}
  Set 
  \[
  F_\s=\{Z\in F: \hat\eng_{\delta,\unm_\s}(f_\s(Z))< C' b_0^{-3\varepsilon}\}
  \]
  Then~\eqref{eq: int of energy for good t} and Chebychev's inequality again imply that $\unm(F\setminus F_\s)\leq  b_0^{\varepsilon}$.

  Altogether, for every $\s\in B'$ and $Z\in F_\s$, we have 
  \[
  \hat\eng_{\delta,\unm_\s}(f_{\s}(Z))=\int\norm{f_\s(Z)-f_\s(Z')}_+^{-\delta}\diff\!\unm_\s(Z')\leq C' b_0^{-3\varepsilon}
  \]
This and Lemma~\ref{lem: truncated energy} imply that for every $\s\in B'$ and $Z\in F_\s$, we have
\[
\unm_\s\Bigl(B_{\R^3}(f_\s(Z), b)\Bigr)\leq C' b_0^{-3\varepsilon}\cdot  b^\delta \quad\text{for all $ b\geq  b_0$}.
\]
The proof is complete. 
\end{proof}

We now turn to the proof of Theorem~\ref{thm: main finitary}. 
\begin{proof}[Proof of Theorem~\ref{thm: main finitary}]
    
We will use the following observation: for all $s\in\R$ and $\s\in\R^{n-2}$, we have  
\begin{equation}\label{eq: decomp pi t'}
\begin{aligned}
    \pi_{s\s}(r_1,\c, r_2)&=r_1+\c\cdot L(s\s)+r_2q(s\s)\\
    &=(1, s, s^2)\cdot f_\s(r_1,\c, r_2). 
\end{aligned}
\end{equation} 

Apply Proposition~\ref{prop: ft preserves dim} with $\varepsilon$ as in the statement of Theorem~\ref{thm: main finitary}. 
Let $B'\subset B$ be as in that proposition, and for every $\s\in B'$, let $F_{\s}$ be as in that proposition. 
Then we have 
\begin{equation}\label{eq: lem ave trun egy proj use}
\unm_\s\Bigl(B_{\R^3}(f_\s(Z), b)\Bigr)\leq C' b_0^{-3\varepsilon}\cdot b^\delta\quad\text{for all $Z\in F_\s$ and $ b\geq b_0$}.
\end{equation}

Let $K_\s= f_\s(F_\s)\subset \R^3$ and let $\rho_\s$ be the restriction of $\unm_t$ to $K_\s$ normalized to be a probability measure. Then~\eqref{eq: lem ave trun egy proj use} and the fact that $\unm(F\setminus F_{\s})\leq  b_0^{\varepsilon}$ imply that   
\begin{equation*}
\rho_\s\Bigl(B_{\R^3}(Y, b)\Bigr)\leq 2C'  b_0^{-3\varepsilon}\cdot b^\delta\quad\text{for all $Y\in K_\s$ and $ b\geq b_0$}.
\end{equation*}
This in particular implies that $K_\s$ and  $\rho_\s$ satisfy the conditions in~\cite[Thm.~B]{MR4549089}, see also \cite{PYZ} and \cite[Thm.\ 2.1]{GGW}. Apply~\cite[Thm.~B]{MR4549089} with $\varepsilon$; thus, there there exists $J_{ b,\s}\subset [0,2]$ with 
\[
\absolute{[0,2]\setminus J_{ b,\s}}\leq \bar C \varepsilon^{-A} b^{\varepsilon}
\]
and for all $s\in J_{ b,\s}$ there is a subset $K_{ b, \s, s}\subset K_{\s}$ with 
\[
\rho_\s(K_{\s}\setminus K_{ b, \s, s}) \leq \bar C\varepsilon^{-A} b^{\varepsilon}
\]
so that for all $Y\in K_{ b, \s, s}$, we have 
\[
\rho_\s\Bigl(\{Y'\in K_\s: \absolute{(1,s,s^2)\cdot (Y-Y')}\leq  b\}\Bigr)\leq \bar C b_0^{-3\varepsilon} \cdot b^{\delta-7\varepsilon};
\]

In view of the definition of $\rho_\s$ and~\eqref{eq: decomp pi t'}, we have the following. For every $\s\in B'$ 
and $s\in J_{ b,\s}$, put $F_{ b, s\s}=F\cap f^{-1}_\s(K_{ b, \s, s})$. Then 
\[
\#(F\setminus F_{ b, s\s})\leq  10\bar C\varepsilon^{-A} b^{\varepsilon}\cdot (\#F),
\]
and for every $Z\in F_{ b, s\s}$, we have 
\begin{equation}\label{eq: F b st}
\#\{Z'\in F_{ b, s\s}: \absolute{\pi_{s\s}(Z)-\pi_{s\s}(Z')}\leq  b\}\Bigr)\leq \bar C b_0^{-3\varepsilon} 
 b^{\delta- 7\varepsilon}. 
\end{equation}

We now use this to complete the proof. Indeed, let $B_ b\subset B$ be the set of $\s\in  B$ 
for which the following holds: there exists $F_{ b, \s}\subset F$ with
\[
\#(F\setminus F_{ b, \s})\leq 100\bar C\varepsilon^{-A} b^{\varepsilon}\cdot (\#F)
\] 
so that for all $Z\in F_{ b, \s}$ we have
\[
\#\{Z'\in F_{ b, \s}: \absolute{\pi_{\s}(Z)-\pi_{\s}(Z')}\leq  b\}\Bigr)\leq \bar C b_0^{-10\varepsilon} 
 b^{\delta}. 
\]
Then~\eqref{eq: F b st} implies that for every $\s'\in B'$ and $s\in J_{ b,\s'}$, we have $s\s'\in B_ b$, so long as $s\s'\in B$. In particular, we conclude that 
\[
\absolute{B\setminus B_ b}\ll \varepsilon^{-A} b^{\varepsilon}.
\] 
Since $b\geq b_0$, this and the definition of $B$, see~\eqref{eq: define B} imply that 
\[
\absolute{B_{\mathbb R^{n-2}}(0,1)\setminus B_ b}\ll \varepsilon^{-A} b^{\varepsilon/10}.
\] 
The proof is complete. 
\end{proof}

\section{Tracking the periodic orbit and constructing a measure on the remaining horosphere}

\begin{prop}\label{prop:prop2}
    Let $0<\eta<\frac{1}{10}\eta_X$ and $0<100\varepsilon <\delta<1$. Suppose there exists $y_0\in X_{2\eta}\cap Y$ and $F\subset B_\mathfrak r (0,1)$ containing 0 so that $\exp(w)y_0\in X_{2\eta}\cap Y$ for every $w\in F$ and
    \begin{equation}\label{Prop2Inequality}
        \sum_{w'\in F\setminus\{w\}} \|w-w'\|^{-\delta}\ll \#F
    \end{equation}
    for all $w\in F$.

    Assume further that $\#F$ is large enough, depending explicitly on $\eta$ and $\varepsilon$: 
    \begin{equation}\label{eqn:sizeoforbits}
        (\#F)^{-\varepsilon}\le \frac{e^{-\ref{C:nondiv}}}{10^5\ref{C:changeinballs}}\eta^{3}
    \end{equation}

    Then there exists $I\subset [0,1]$, some $b_1>0$ with 
    \begin{equation}\label{eqn: size of b1}
        (\#F)^{-\frac{n-\delta+(2+n)\varepsilon}{n-\delta +(3+n)\varepsilon}}\le b_1\le (\#F)^{-\varepsilon},
    \end{equation}
    and some $y_2\in X_{2\eta}\cap Y $ so that both of the following hold true:
    \begin{enumerate}
        \item The set $I$ supports a probability measure $\rho$ which satisfies
        $$\rho(J)\le C_\varepsilon (b_0/b_1)^{-\varepsilon}\cdot|J|^{\delta - (3+n)\varepsilon}$$
        for all intervals $J$ of size $|J|\ge\frac{b_0}{b_1}$ where $b_0=(\#F)^{-1}$ and where $C_\varepsilon$ depends on absolute constants, $\varepsilon$, and $n$.
        \item There is an absolute constant $C$, so that for all $r\in I$, we have
        $$v_ry_2\in B_{Cb_1} ^G Y.$$        
    \end{enumerate}
\end{prop}

\begin{proof}
By Equation \ref{eqn:sizeoforbits}, we assume throughout the proof that 
$$        (\#F)^{-\varepsilon}\le \frac{e^{-\ref{C:nondiv}}}{10^5\ref{C:changeinballs}}\eta^{3}$$

By Proposition \ref{prop:input}, there exists $y_0\in Y$  and a finite subset $F\subset B_\mathfrak{r}(0,1)$ that satisfies Equation (\ref{Prop2Inequality}) and so that $\exp(w)y_0\in Y$ for all $w\in F$. In particular, by our choice of $F$, we have $\|w\|\le\beta_0$ for all $w\in F$. Moreover, 
    Equation \eqref{Prop2Inequality} allows us apply Lemma \ref{lem:Localize}. Choose $w_0\in F$, $b_1>0$, and $F'\subset B_\mathfrak r(w_0,b_1)\cap F$ as in Lemma \ref{lem:Localize} where 
    \begin{equation}\label{eqn: bound on b1}
        (\# F)^{-\frac{n-\delta+(2+n)\varepsilon}{n-\delta+(3+n)\varepsilon}}\le b_1\le  (\# F)^{-\varepsilon}.
    \end{equation}
    By using a pigeonhole argument, we may replace $w_0$ with a different point in $F$  (at the cost of potentially increasing $C'$) and assume that $F'\subset B_\mathfrak r (w_0,b_1/2\dim(G)\ref{C:normalform})\cap F$.
Now, due to Lemma \ref{lem: normal form}, for each $w'\in F'$, there exists $h\in H$ and $w\in \mathfrak{r}$ so that
    \begin{equation}\label{eqn: conditions normal form}
        \begin{split}
        h\exp(w) = \exp(w')\exp(-w_0),\\
        \|h-I\|\le \frac{b_1 \eta^2}{\dim(G)},\text{ and}\\
        \|w\|\le \frac{b_1 }{\ref{C:normalform}\dim(G)}. \end{split}
    \end{equation}

    Set 
    \begin{equation}\label{eqn: recentered fiber}
        E=\{w\in\mathfrak{r}: \exists h\in H, w'\in F' \text{ so that }h,w,w_0,w' \text{ satisfy}~\eqref{eqn: conditions normal form}\}
    \end{equation}

    We claim that $E$ has dimension inherited from $F'$ but with a larger multiplicative constant i.e. that
    \begin{equation}\label{eqn: E has dimension}
        \frac{\#(E\cap B(w,b))}{\#E}\le \hat C (b/b_1)^{\delta - (3+n)\varepsilon}
    \end{equation}
    for all $w\in\mathfrak{r}$ and $b\ge (\#F)^{-1}$ where $\hat C\le 2C'$. This is just a consequence of the fact that the map $f:B_\mathfrak{r}(0,\beta_0/2)\to B_\mathfrak{r}(0,\beta_0)$ defined by $f(w')=w$ where $$h\exp(w)=\exp(w')\exp(w_0)$$ is a diffeomorphism (see Lemma \ref{lem: normal form}). In particular, $D_{w'}(f^{\pm1})$ is invertible for all $w'\in B_\mathfrak{r}(0,\beta_0/2)$ and so $\#f(E)=\#E$ and 
    $$\#\left(B_\mathfrak{r}(\bar w,\varrho)\cap f(E)\right)\le \#\left(B_\mathfrak{r}(f^{-1}(\bar w),2\varrho)\cap E\right)$$
    for all $\varrho\le\beta_0/2$ which implies the claim.

    Let $y_1:=\exp(w_0)y_0$ and let $w'\in F'$. Then there exists $h$ and $w$ that satisfy ~\eqref{eqn: conditions normal form} and we have
    \begin{equation}\label{eqn: calcu}
     h\exp(w)y_1 =\exp(w')\exp(-w_0)\exp(w_0)y_0= \exp(w')y_0\in Y  
    \end{equation}
    by construction of $F$ and using that $w'\in F$.

    Moreover, we need the following lemma that allows us to assume the size of certain elements cannot get too small. We postpone the proof until the proposition is proved.
    \begin{lem}\label{lem:expanding part not too small}
        There exists $\s\in B_1 ^U$ and a subset $$\bar E\subset \Ad(u_\s)E\cap \{w\in B_\mathfrak{r}(0,1):|w|^+\ge10^{-3}\|w\|\}$$
        so that $\#\bar E\ge\#E/4$ where $Z\mapsto Z^+$ is the projection onto the expanding direction of $a_t$.
    \end{lem}

   Due to this lemma, we replace $y_1$ by $u_\s y_1$ for some $\s\in B_1^U$ and $E$ by a subset $\bar E$ with $\#\bar E\ge\#E/4$ (which we continue to denote by $E$) to ensure that 
   \begin{equation}\label{eqn: lower bound on expanding direction}
       E\subset  \{w\in B_\mathfrak{r}(0,1):|w^+|\ge10^{-3}\|w\|\}.
   \end{equation}
    Note that \eqref{eqn: E has dimension} holds for the new $E$ with $4\hat C$, but we will suppress the factor 4.

    \textbf{Estimates on the size of elements}

    Let $t=|\log(b_1)|$. By \eqref{eqn: calcu}, since $a_tu_\s\in H$, we have 
    \begin{equation}\label{eqn: calculation2}
        a_tu_\mathbf{s}h\exp(w)y_1\in Y
    \end{equation}
    and for every $\mathbf{s}\in B_1 ^U$ and $w\in E$, we have, by Equation~\eqref{eqn:Adjointaction},
    \begin{equation}\label{eqn: expanding translate of lie alg}
        \|\Ad(a_tu_\s)w\|\le \frac{1}{b_1}\|\Ad(u_\s)w\|\le\frac{3}{b_1}\|w\|\le \frac{3}{\dim(G)\ref{C:normalform}}\le 1.
    \end{equation}
    Additionally, we note that $(a_tu_
    \s)h(a_tu_\s)^{-1}\in H$.

    Now, thanks to Equation~\eqref{eqn: lower bound on expanding direction}, if $\|\s\|\le 10^{-4}$, then the reverse triangle inequality yields
    $$|(\Ad(u_\s)w)^+|=|r_1+\s\cdot\c+r_2\|\s\|^2/2|\ge|r_1|-|\s\cdot\c+r_2\|\s\|^2/2|\ge (10^{-3}-2\cdot10^{-4})\|w\|\ge 10^{-4}\|w\|.$$     
    In particular, for $10^{-5}\le\|\s\|\le 10^{-4}$, the upper bound for $\|w\|$ above, shows we have that the contracting direction of $\Ad(a_tu_\s)w$ satisfies $\le e^{-2t}10^4 \|(\Ad(a_tu_\s)w)^+\|$ and the neutral satisfies $\le 2\cdot10^4e^{-t}\|(\Ad(a_tu_\s)w)^+\|$.

    In conclusion, we have  for $\|\s\|\le 10^{-4}$, 
    $$a_tu_\s h\exp(w).y_1  = (a_tu_
    \s)h(a_tu_\s)^{-1}\cdot g\cdot \exp(Z((\Ad(a_tu_\s)w)^+,\mathbf{0},0).a_tu_\s y_1
    $$
    where, due to the Baker--Campbell--Hausdorff formula and Equation~\eqref{eqn: expanding translate of lie alg}, $g\in G$ satisfies
    \begin{equation}
        \|g-I\|\le 4\cdot 10^4 b_1
    \end{equation}
    with an absolute implied constant.  Indeed, $[(\Ad(a_tu_\s)w),-(\Ad(a_tu_\s)w)^+]=[(\Ad(a_tu_\s)w)-(\Ad(a_tu_\s)w)^+,-(\Ad(a_tu_\s)w)^+]$ and we have shown a few paragraphs before that the contracting and neutral directions of $(\Ad(a_tu_\s)w)$ have size less than $2\cdot10^4 b_1$.

    Hence, by Equation \ref{eqn:changingballs}, we  have $g\in B_{\ref{C:changeinballs}4\cdot10^4b_1} ^G$.

    As such
    \begin{equation}\label{eqn:expanding close to Y}
        \exp(Z((\Ad(a_tu_\s)w)^+,\mathbf{0},0).a_tu_\s y_1\in B^G _{Cb_1} Y
    \end{equation}
    where $C=\ref{C:changeinballs}4\cdot10^4$ is an absolute constant.

     We now choose a particular $s$ to define the set $I$. Recall that by our choice of $t=|\log(b_1)|$ and that 
     $$b_1\le         (\#F)^{-\varepsilon}\le \frac{e^{-\ref{C:nondiv}}}{10^5\ref{C:changeinballs}}\eta^{3}.
$$
In particular, $t$ is sufficiently large that we may apply Proposition \ref{prop:nondiv} to the point $y_1$ and $B=\{\s\in B_1^U:10^{-5}\le\|s\|\le10^{-4}\}$, $\eta= \min\{|B|,\alpha^2\}$, and $\alpha\le.1\ref{C:nondiv}^{-1}$ to deduce that
$$B'':=\{\s\in J: a_tu_\s y_1 \in X_\eta\}$$
has $|B''|>.9|B|$.

Now apply Theorem \ref{thm:ProjSO(n,1)} to $E\subseteq B_\mathfrak r (0,b/b_1)$ (see Equation \eqref{eqn: recentered fiber} and \eqref{eqn: lower bound on expanding direction}), $B=\{\s\in B_1^U:10^{-5}\le\|s\|\le10^{-4}\}$, coarse dimension $\delta-(3+n)\varepsilon$ and $\varepsilon\le\frac{\delta}{2(3+n)}$ to get some $B'\subset B_1 ^U$ with $|B'|\ge .99|B|$. 

Fix some $\s\in B'\cap B''$. Put $y_2:=a_tu_\s y_1$. By definition of $B''$ we have $y_2\in X_\eta$ and, by Equation \ref{eqn:expanding close to Y}, we have
$$\exp(e^t Z((\Ad(u_\s)w)^+,\mathbf 0,0))y_2\in B^G _{Cb_1}Y.$$

Define $I:=\{e^t\xi_\s(w):w\in E_{\s} \}$. We claim the proposition holds with $y_2$, $I$, and $b_1$.

By construction we have the claimed bound from Equation \ref{eqn: size of b1} on $b_1$ and item (2) follows from Equation \ref{eqn:expanding close to Y}. As such, it only remains to prove item (1). 

Let $\mathfrak{b} \ge e^t(\#F)^{-1}$ and $b_0=(\#F)^{-1}$.

Let $\rho$ be defined as
$$\rho(A)=\frac{\#\{w\in E_{\s}:e^t\xi_\s(w)\in A\}}{\# E_{\s}}$$
for $A$ an interval.
That is, $\rho$ is the pushforward of the normalized counting measure on $E_\s$ under the map $w\mapsto e^t\xi_\s(w)$.

Let $w\in E_\s$ and $r=e^t\xi_\s(w)$. By Theorem \ref{thm:ProjSO(n,1)} and the fact that for every $\mathfrak{b}\ge e^t(\#F)^{-1}$ we have
$\#E_\s\ge .9(\#E)$, we conclude
\begin{align}\label{eqn:measure on remaining horosphere}
    \rho(\{r'\in I:|r-r'|\le \mathfrak{b}\})&=\frac{\#(\{w'\in E_\s:|\xi_\s(w')-\xi_\s(w)|\le e^{-t}\mathfrak{b}\}}{\#E_\s}\\\nonumber&\le C_\varepsilon(b_0/b_1)^{-\varepsilon}(e^{-t}\mathfrak{b}/b_1)^{\delta - (3+n)\varepsilon}\\\nonumber&\le C_\varepsilon (b_0/b_1)^{-\varepsilon}\mathfrak{b}^{\delta - (3+n)\varepsilon}
\end{align}
where $C_\varepsilon$ depends on $\hat C$ from Equation \ref{eqn: E has dimension}, $\varepsilon$, and $n.$

This estimate finishes item (1) and thus concludes the proof.    
\end{proof}

\begin{proof}[Proof of Lemma \ref{lem:expanding part not too small}]
    Throughout we assume $w$ has coordinates $(r_1,\c,r_2)$.

If $\{w\in E: w^+\ge 10^{-3}\|w\|\}$ has more than $\#E/4$ elements, then done with $\s=\mathbf 0$. Else we assume that $\#\{w\in E: w^+\ge 10^{-3}\|w\|\}<\#E/4$ and so $.75\#E\le \#\hat E$ where $\hat E:=\{w\in E: w^+\le 10^{-3}\|w\| \}$.

Case 1: $\#\{w\in \hat E: \|\c\|\ge.1\|w\|\}\ge\#E/4$.
Then, there is at least one coordinate in $\c$, say $c_i$, with $|c_i|\ge.1\|w\|$. Let $\s$ be the zero vector except in the $i$th entry. Then
\begin{align*}
    .1|s_i|\|w\|&\le |c_is_i|=\|\c\cdot\s\|\\
    &\le |r_1+\c\cdot\s+r_2\|s\|^2/2|+|r_1|+|r_2|\|\s\|^2/2\\
    &\le |r_1+\c\cdot\s+r_2\|s\|^2/2|+\|w\|10^{-3}+\|w\||s_i|^2/2
\end{align*}
and so for $s_i=.1$, we get $$.004\|w\|\le|r_1+\c\cdot\s+r_2\|s\|^2/2|.$$
Indeed, $(.1|s_i|-10^{-3}-|s_i|^2/2|)\|w\|\le |r_1+\c\cdot\s+r_2\|s\|^2/2|$

Case 2: If $\#\{w\in \hat E: \|\c\|\ge.1\|w\|\}<\#E/4$, then $.5\#E<\{w\in\hat E:\|\c\|\le .1\|w\|\}$. In this case we have $\|w\|=|r_2|$ and using a similar argument as in case 1 with $\s$ zero everywhere except in a coordinate where $\|c\|$ is attained in which case $s_i$ is $.75$ gives the desired conclusion.
\end{proof}

%%%%%%%%%%%%%%%%%%%%%%%%%%%%%%%%%%%%%%%%%%%%%%%%%%%%%%%%%%%%%%%%%%%%%%%%%%%%%%%%

\section{Proof of the Main results}\label{section: proofs}
We prove Theorem~\ref{Thm:EffDen} and Theorem~\ref{thm:count}.

\begin{proof}[Proof of Theorem \ref{Thm:EffDen}]
    Let $x\in X_{\varrho}$ for some $\varrho\in (0,.1\eta_X)$.  Choose $f_{\varrho,x}$ supported on $B^G _{.1\varrho}x$ with $\int f_{\varrho,x} \,dm_X=1$ and with $\mathcal S(f_{\varrho,x})\le \varrho^{-N}$ for some absolute $N$. 
    
    Recall, by Proposition \ref{prop:prop2}, there exists some $y_2\in Y\cap X_{2\rho}$,  $b_1>0$ satisfying Equation \eqref{eqn: size of b1}, and an interval $I\subset [0,1]$ so that $I$ supports a probability measure $\rho$ which satisfies
        $$\rho(J)\le C_\theta (b_0/b_1)^{-\theta}\cdot|J|^{\delta - (3+n)\theta}$$
        for all intervals $J$ with $|J|\ge\frac{b_0}{b_1}$ where $b_0=\#F^{-1}$ where $C_\theta$ depends on absolute constants, $\theta$, and $n$,  and $\theta\le\frac{\delta}{2(3+n)}$. 
     Moreover, we have $v_ry_2\in B_{Cb_1} ^GY$ for all $r\in I$ and absolute $C>0.$ Note we have replaced some of the notation from Proposition \ref{prop:prop2} in anticipation of using Theorem \ref{thm:VenkateshSO(n,1)}.

     Let $b=\frac{b_0}{b_1}$ and $t=|\log(b)|/4$. Let $\theta=\frac{\min\{\delta, \varepsilon_0\}}{\ell(3+n)}$ where $\varepsilon_0$ is from Theorem  \ref{thm:VenkateshSO(n,1)} and $\ell\ge5$ is chosen so that $\frac{\ref{K:Venk}}{2^{\star}}\le\theta\le\frac{\ref{K:Venk}}{2^{\star-1}}$. Choose $\delta = 1-\varepsilon_0/4$. By Theorem \ref{thm:VenkateshSO(n,1)}, we have
    \begin{equation}\label{eqn:using venkatesh}
       \left| \int_{\s\in B_1^U}\int_{r\in[0,1]}f_{\varrho,x}(a_{t}u_\s v_r y_2)\,d\rho(r)\,d(\s)-1\right|\ll \left( C_\theta (b_0/b_1)^{-\theta}\right)^{1/2}\mathcal S(f_{\varrho,x})e^{-\ref{K:Venk} t}\varrho^{-L/2}
    \end{equation}    
    for an absolute $L$.
    
    Choose $\varrho = e^{-\ref{K:Venk} t/2N}$ (and so $\varrho = b^{\ref{K:Venk}/8N}$) to get the right-hand side of Equation \eqref{eqn:using venkatesh} 
    $$\le C_\theta^{1/2} b^{-\theta/2}e^{-\ref{K:Venk} t/2}\varrho^{-L/2}= C_\theta^{1/2} b^{-\theta/2}b^{\ref{K:Venk} /8}\varrho^{-L/2}\ll_{\theta,\varrho} b^{\ref{K:Venk}/8}$$   
     where we have absorbed the dependence on $\varrho$ and $C_\theta$ into the implicit constant and used $\frac{\ref{K:Venk}}{2^{\ell}}\le\theta\implies 0\le\ref{K:Venk}/8-\theta/2\le \ref{K:Venk}/8$. In particular, for periodic orbits $Y$ with sufficiently large volume, the right hand side of the previous equation is $<1/2$.  Thus $a_{t}u_\s v_r y_1\in \text{supp}(f)$ for some $r\in I$ and $\s\in B_1^U$. In particular, $a_{t}u_\s v_r y_2\in B_{.1\varrho} ^G x$ for some $r\in I$ and $\s\in B_1^U$. 

    Let $\ref{K:den}=\ref{K:Venk}/8N$. Then the above implies 
    \begin{equation}\label{eqn:Venkatesh gives close to point}
        d_X(x,a_tu_\s v_ry_2)\ll b^{\ref{K:den}}
    \end{equation}
    for all $x\in X_{b^{\ref{K:den}}}$.

    By Proposition \ref{prop:prop2}, we have $v_ry_2\in B_{Cb_1} ^GY$ for all $r\in I$ and so 
    $$d_X(v_ry_2,Y)\le d_G(e,g)\ll \|g-I\|\ll b$$
    for all $r\in I$. For each $\s\in B_1 ^U$, $\|u_\s gu_\s ^{-1}-I\|\ll b$ and so $d_X(u_\s v_ry_2,Y)\ll b.$

     Note that if $\|g'-I\|\ll b$, then $\|a_tg'a_{-t}-I\|\ll bb^{-1/2}=b^{1/2}$  since $e^t=b^{-1/4}.$ As such,    \begin{equation}\label{eqn: cor of prop 2}
     d_X(a_tu_\s v_ry_2,Y)\ll b^{1/2}   
    \end{equation}

    Then Equation \eqref{eqn:Venkatesh gives close to point} and \eqref{eqn: cor of prop 2} imply
    $$d_X(x,Y)\ll b^{\ref{K:den}} $$
    for all $x\in X_{b^{\ref{K:den}}}$. 

    Thus,  the fact that $\#F\gg\vol(Y)$ imply that 
    $$d_X(x,Y)\ll \vol(Y)^{-\ref{K:den}} $$ for all $x\in X_{\vol(Y)^{-\ref{K:den}}}$ and where the implicit constant depends only on $\dim(G)$ and $\ref{K:den}$.

    As remarked after Theorem \ref{Thm:EffDen}, if $\Gamma$ is a congruence subgroup, then $\ref{K:den} = \ref{K:den}(\kappa_X)$ is absolute.   
\end{proof}

\begin{proof}[Proof of Theorem \ref{thm:count}]
Any geodesic hyperplane in $M$ lifts to a periodic orbit of $H=\SO(n-1,1)$ in the frame bundle $X=G/\Gamma$ and so we work with the lifts. Our strategy is to use the geometry of manifolds that come from the gluing construction to obtain an upper bound on the volumes of lifts of totally geodesic hyperplanes and invoke Theorem 5 of \cite{SS22} (quantitative bound on bounded volume orbits) to count the number of lifts. As such, we will often impose a lower bound on the volume of lifts with which we are working with since those that are already bounded above can be controlled.

By the thick-thin decomposition, we have that $X-X_{\eta_X}$ is a disjoint union of finitely many cusps. Let $M_0$ denote the image of $X_{\eta_X}$ in $M$ so that $M-M_0$ decomposes into a disjoint union of finitely many (possibly none) cusps.

We will work with periodic orbits $Hx$ of sufficiently large volume so that 
$$\ref{C:main}\vol(Hx)^{-\ref{K:den}}<\eta_X.$$
This implies that points in the compact set $X_{\eta_X}\subset X_{\vol(Hx)^{-\ref{K:den}}}$ satisfy effective density (Theorem \ref{Thm:EffDen}). Indeed, $\ref{C:main}>1$. 

Let $r>0$ be such that for $i=1,2$, we have some $x_i \in X_{\eta_X}\subset X_{\vol(Hx)^{-\ref{K:den}}}$ so that $B_{r} ^G.x_i$ projects to the interior of $N_i\cap M_0$. In lieu of Theorem 4 (Quantitative Isolation of closed orbits) of \cite{SS22},  there exists some constant $D = D(\dim(G))>0$ such that the radius satisfies $r\gg \vol(\Sigma)^{-2D}$. Theorem 4 needs a compact set and this is another reason why we needed to restrict to $X_{\eta_X}$. Moreover, in the cusp ($X_{\eta_X}^c$), the radius would have to be much smaller.

We now consider periodic orbits $Hx$ of sufficiently large volume so that 
\begin{equation}\label{eqn: count bound1}
    \ref{C:main}\vol(Hx)^{-\ref{K:den}}\le .5\min\{r,\eta_X\}
\end{equation}
Consequently, Theorem \ref{Thm:EffDen} yields that $Hx\cap B_r ^G.x_i\ne \emptyset$ for $i=1,2$ and so the corresponding plane crosses $\Sigma.$ By angular rigidity of \cite{FLMS} (Theorem 4.1), the corresponding plane crosses orthogonally. Denote this plane by $S$.

On the other hand, the since the hyperplane $S$ is crossing orthogonally, then it cannot be too dense. Indeed, Proposition 5.1 of \cite{FLMS} constructs an open set $O$ of the unit tangent bundle such that the projection into the 1-neighborhood of $M_0$ and does not intersect. This is just constructing an open set of angles away from the normal vector of the plane $S$. 

Let $\rho>0$ and $x\in X$ be so that $B_{\rho} ^G.x$ projects into $O$. We note that by Quantitative Isolation of closed orbits of \cite{SS22} once more, we have $\rho\gg\vol(\Sigma)^{-2D}.$ Lastly, by the previous paragraph, we have $Hx\cap B_\rho ^G.x=\emptyset$. Thus, by effective density applied to $Hx$, we must have $$\ref{C:main}\vol(Hx)^{-\ref{K:den}}> .5\rho.$$

Combining this with Equation~\eqref{eqn: count bound1}, we get
$$\vol(Hx)\le \left(\frac{2 \ref{C:main}}{\min\{\eta_X,r,\rho\}}\right)^{{1/\ref{K:den}}}.$$

As such, we can invoke Theorem 5 of \cite{SS22} (quantitative bound on bounded volume orbits) to bound the number of lifts to deduce
$$\#\left\{Y:Y \text{ is a closed } S \text{-orbit and }\vol(Y)\le\left(\frac{2 \ref{C:main}}{\min\{\eta_X,r,\rho\}}\right)^{{1/\ref{K:den}}}\right\}\ll \left(\frac{2 \ref{C:main}}{\min\{\eta_X,r,\rho\}}\right)^{{D'/\ref{K:den}}}$$

where the implicit constant is absolute and $D'$ depends on the dimension of $G$.

By recalling that $\ref{C:main}$ depends on $\kappa_X, \vol(X),$ and the minimum injectivity radius of points in $X_{\eta_X}$ and that $\ref{K:den} =O(\kappa_X ^2)$ depends on $\kappa_X$, and that $r,\rho\gg\vol(\Sigma)^{-D}$ we get the desired conclusion.
\end{proof}

\begin{remark}\label{rmk:extending}
 We remark that angular rigidity holds for a slightly larger class of non-arithmetic hyperbolic manifolds \cite{FLMS}. Briefly, the Gromov--Piatestski-Shapiro hybrid manifolds are built from two subarithmetic non-commensurable pieces and one can generalize this to more pieces in such a way that one obtains a non-arthmetic manifold. 
 
 Thus, the conclusion of Theorem \ref{thm:count} holds for this class of hyperbolic manifolds. Indeed, by using the terminology of \cite{FLMS}, once we can guarantee that a hypersurface passes between two adjacent building blocks that are arithmetic and dissimilar, then angular rigidity applies and one can finish the argument exactly as in Theorem \ref{thm:count}. Examples include those by Raimbault \cite{10.1093/imrn/rns151} and Gelander--Levit \cite{Commclasses}.
\end{remark}
%%%%%%%%%%%%%%%%%%%%%%%%%%%%%%%%%%%%%%%%%%%%%%%%%%%%%%%%%%%%%%%%%%%%%%%%%%%%%%%%%%%%%%%%%%%%

\section{Appendix: Effective equidistribution over measures of large dimension}\label{section: Appendix A}
In the appendix we prove a slightly more general version of Theorem \ref{thm:VenkateshSO(n,1)} so that it can be utilized outside of this paper. That is, we show that, in the context of simple Lie groups, we have equidistribution in the presence of measures on subgroups that behave horospherically and that are almost full. 
\begin{thm}\label{thm:Appendix}
    Let $G$ be a simple Lie group with finite center, $\Gamma$ a lattice of $G$, and $X=G/\Gamma$. Let $X_\eta=\{x\in X:\inj(x)\ge\eta\}$. Suppose
    $\mathbb R^d \simeq N\subset G$ is a unipotent subgroup equipped with the Euclidean inner product and we have a subgroup $U \simeq \mathbb R^k$ for $k\ge1$. Set $V=U^\perp$ and let $A=\{a_t:t\in \mathbb R\}$ be a 1-parameter subgroup of $G$ satisfying $$a_tn_\mathbf{z}a_{-t}=n_{e^t\mathbf{z}}.$$ 
    Let $[0,1]^d$ denote $\{u_\mathbf{s}v_{\mathbf {r}}\in N:(\mathbf{s},\mathbf{r})\in[0,1]^k\times[0,1]^{d-k}\}$.
    Assume the following holds
    \begin{itemize}
        \item (Effective equidistribution of the unstable foliation) There exists $\constK\label{K:effec}=\ref{K:effec}(\Gamma,\kappa_X)$ and $L>0$ such that the following holds. Let $0<\eta<1$, $t>0$, and $x\in X_\eta$. Then for every $f\in C_c^\infty(X)+\mathbb C \cdot1$,
        $$\left|\int_{[0,1]^d}f(a_tn_\mathbf{z}x)dm_N(\mathbf{z})-\int_X f(x)dm_X(x)\right|\ll \mathcal S(f)e^{-\ref{K:effec} t}\eta^{-L}$$
        where $m_N$ is the Haar measure on $N$ normalized so that $[0,1]^d$ has measure 1, and the implicit constants are all absolute.
        \item Let $\rho$ be a probability measure on $[0,1]^{d-k}\subseteq \mathbb R^{d-k}\simeq V$ satisfying 
        $$\rho(B(\mathbf{r},b))\le Cb^{\delta}$$
        for some $b>0$.
    \end{itemize}
    Then, there exists $\kappa=\kappa(\kappa_X)$ and $\varepsilon=\varepsilon(\kappa_X)$ such that if $\delta>d-k-\varepsilon$, then for $|\log(b)/4|\le t\le |\log(b)/2|$ we have
    $$\left|\int_{[0,1]^{d-k}}\int_{[0,1]^k}f(a_t u_{\mathbf{s}}v_{\mathbf{r}}x)d\mathbf{s}d\rho(\mathbf{r})-\int_Xf(x)dm_X(x)\right|\ll C^{1/2}S(f)e^{-\kappa t}\eta^{-L/2}.$$ Here $\kappa_X$ comes from the decay of correlations in $X$, see Proposition \ref{Cor:Decay of Cor}. 
\end{thm}

Notice that the theorem needs effective equidistribution with respect to the Haar measure on $N$ as an assumption. When $N$ is the horospherical subgroup of a 1-parameter diagonalizable element, then effective equidistribution of expanding unipotents with respect to the Haar measure on the unstable foliation is known by \cite{MR1359098} Proposition 2.4.8. See also \cite{MR4277274} for effective equidistribution of expanding translates in more general homogeneous spaces. Also related are the results of \cite{MR3556973} and \cite{MR4372213}  where they prove effective equidistirbution results of expanding curves contained in $G$ and products of $G$. A version of Theorem \ref{thm:Appendix} in the setting of Teichm\"{u}ller dynamics can be found in \cite{S23}.

Thus, Theorem \ref{thm:VenkateshSO(n,1)} is a special case of Theorem \ref{thm:Appendix}.

\begin{proof}[Proof of Theorem \ref{thm:VenkateshSO(n,1)}]
Apply Theorem \ref{thm:Appendix} to $G=\SO^\circ(n,1)$ and $N$ the horospherical subgroup of $a_t = \text{diag}(e^{t},\ldots,e^{-t})$ and note that the existence of the required probability measure follows from Proposition \ref{prop:prop2} and $[0,1]^{n-2}\subset  [-1,1]^{n-2}\simeq B_1^U$.
\end{proof}

Before beginning the proof of Theorem \ref{thm:Appendix}  we record a result that we need on the decay of correlations.

\begin{prop}\label{Cor:Decay of Cor}(Corollary 2.4.4 of \cite{MR1359098})
    Let $d$ be a fixed right $G$-invariant metric on $G$. There exists a positive constant $\kappa_X$ on $X=G/\Gamma$ so that
    $$\left|\int \varphi(gx)\psi(x)\,dm_X-\int_X\varphi\,dm_X \int_X\psi\,dm_X\right|\ll e^{-\kappa_Xd (e,g)}\mathcal S(\varphi)\mathcal S(\psi)$$
    for every $f\in C_c ^\infty(X)+\mathbb C \cdot1$.
\end{prop}

The Sobolev norm is assumed to dominate $\|\cdot\|_\infty$ and the Lipschitz norm $\|\cdot\|_\textrm{Lip}$. Additionally, $\mathcal S(f\circ g)\ll\|g\|^\star \mathcal S(f)$ where the implied constants are absolute.
%%%%%%%%%%%%%%%%%%%%%%%%%%%%%%%%%%%%%%%%%%%%%%%%%%%%%%%%%%%%%%%%
\begin{proof}[Proof of Theorem \ref{thm:Appendix}]
 By choosing $N\in\mathbb N$ so that $\frac{1}{N}\le b<\frac{1}{N-1}$, we suppose that $b=1/N$.   Additionally, we assume $f$ is Lipshitz since we can always approximate by such function. Lastly, we let $\delta>d-k-\varepsilon$ for $\varepsilon$ that we will optimize later.

 For a vector $\mathbf{k}\in\{0,\ldots, N-1\}^d$, define the $b$-ball $I_\mathbf{k}$ set of directions one can move in $[0,1]^{d-k}\subset \mathbb R^{d-k}\simeq V $ by
$$I_\mathbf{k}=\{v_{\mathbf{r}}:\mathbf{r}=(r_j),r_j\in [k_jb,(k_j+1)b)^{d-k}, \text{ for each }j\}.$$
Define $c_\mathbf{k}=\rho(I_\mathbf{k})$. Notice that since the $b$-balls $(I_\mathbf{k})_{\mathbf{k} \in\{0,\ldots, N-1\}^d}$ are disjoint and their union is $[0,1]^{d-k}$  that $\sum_{\mathbf{k} \in\{0,\ldots, N-1\}^d}c_\mathbf{k}=1$. 

For each $\mathbf{k} \in\{0,\ldots, N-1\}^d$, let
$$B_\mathbf{k} =\left\{u_{\mathbf{s}}v_{\mathbf{r}}:\mathbf{s}\in [0,1]^k,\mathbf{r}=(r_j),r_j\in\left(k_jb,k_jb+\frac{b}{4}\right)\text{ for each }j=1,\ldots, d-k \right\}$$
denote a thickening to the full unstable leaf.
The sets $B_\mathbf{k}$ continue to be disjoint for different indices $\mathbf{k}$. Then,
\begin{align*}
    \int_{[0,1]^{d-k}}\int_{[0,1]^k} f(a_t u_{\mathbf{s}}v_{\mathbf{r}}x)\,d\mathbf{s}\,d\rho(\mathbf{r}) &=\sum_\mathbf{k}\int_{I_\mathbf{k}}\int_{{[0,1]^k}} f(a_t u_{\mathbf{s}}v_{\mathbf{r}}x)\,d\mathbf{s}\,d\rho(\mathbf{r})     %\\&=\sum_\mathbf{k}\int_{\mathbf{r}\in[0,b)^{d-k}}\int_{[0,1]^k}f(a_t u_{\mathbf{s}}v_{b\mathbf{\mathbf{k}+\mathbf{r}}}x )\,d\mathbf{s}\,d\rho(\mathbf{r}) .
\end{align*}

Thus,
\begin{align*}
    &\left|\int_{[0,1]^{d-k}}\int_{[0,1]^{k}} f(a_t u_{\mathbf{s}}v_{\mathbf{r}}x)\,d\mathbf{s}\,d\rho(\mathbf{r}) - \sum_\mathbf{k}c_{\mathbf{k}}\int_{[0,1]^{k}}f(a_t u_{\mathbf{s}}v_{b\mathbf{k}}x)\,d\mathbf{s}\right|\\
    = &\sum_\mathbf{k}\int_{I_\mathbf{k}}\int_{[0,1]^{k}}\left|f(a_t u_{\mathbf{s}}v_{\mathbf{r}}x) - f(a_t u_{\mathbf{s}}v_{b\mathbf{k}}x)\right|\,d\mathbf{s}\,d\rho(\mathbf{r}).
\end{align*}

We now compare the difference between $f(a_t u_{\mathbf{s}}v_{\mathbf{r}}x)$ and $ f(a_t u_{\mathbf{s}}v_{b\mathbf{k}}x)$ by recalling that we assume $f$ is a Lipschitz function.  We have $\|\mathbf{r}-
b\mathbf{k}\|\le b$ for $\mathbf{r}\in [0,b)^{d-k}$ and $t\le |\log(b)|/2$ by assumption and so 
\begin{align*}
\left|f(a_t u_{\mathbf{s}}v_{\mathbf{r}}x) - f(a_t u_{\mathbf{s}}v_{b\mathbf{k}}x)\right|&\le \mathcal S(f\circ a_t u_{\mathbf{s}})\textrm{d}(v_{\mathbf{r}}x,v_{b\mathbf{k}}x)\\
&\ll \mathcal S(f) e^{t} \|\mathbf{r}-
b\mathbf{k}\|\\
&\le \mathcal S(f)b^{-1/2} b = \mathcal S(f) b^{1/2}.
\end{align*}
Thus, we obtain
\begin{align*}
 &\left|\int_{[0,1]^{d-k}}\int_{[0,1]^{k}} f(a_t u_{\mathbf{s}}v_{\mathbf{r}}x)\,d\mathbf{s}\,d\rho(\mathbf{r}) - \sum_\mathbf{k}c_{\mathbf{k}}\int_{[0,1]^{d}}f(a_t u_{\mathbf{s}}v_{b\mathbf{k}}x)\,d\mathbf{s}\right|\\
\\&\ll \sum_\mathbf{k}\int_{I_\mathbf{k}}\int_{[0,1]^{k}}\mathcal S(f) b^{1/2}\,d\mathbf{s}\,d\rho(\mathbf{r})=\mathcal S (f)b^{1/2}
\end{align*}
and so it suffices to understand the behavior of $f$ on the discrete points $b\mathbf{k} .$
Let 
$$\varphi =  \left(\frac{b}{4}\right)^{-(d-k)}\sum_{\mathbf{k}}  c_\mathbf{k}\mathbb{1}_{B_\mathbf{k}}$$
where the constant in front reflects the reciprocal of the volume of $B_\mathbf{k}$.

Thus, by our dimension bound on $\rho$ and the disjointness of $B_{\mathbf{k}}$, we have the following pointwise bound \begin{align*}
|\varphi|&\le 4^{d-k}b^{-(d-k)}\sum_{\mathbf{k}}c_{\mathbf{k}}\mathbb{1}_{B_\mathbf{k}}<C4^{d-k}b^{-(d-k)}b^{\delta}\mathbb{1}_{\cup_\mathbf{k} B_\mathbf{k}}\\
&<C4^{d-k}b^{-(d-k)}b^{d-k-\varepsilon}\mathbb{1}_{\cup_\mathbf{k} B_\mathbf{k}}\le C4^{d-k} b^{-\varepsilon}
\end{align*}

We have, by noting that $\,d\mathbf{z} = \,d\mathbf{r}\,d\mathbf{s}$,
\begin{align*}
  &\left|\sum_\mathbf{k}c_{\mathbf{k}}\int_{[0,1]^k} f(a_tu_{\mathbf{s}}v_{b\mathbf{k}})\,d\mathbf{s}-\int_{N}\varphi(n(\mathbf{z}))f(a_t n(\mathbf{z}))\,d\mathbf{z}\right|\le\\
  &\sum_\mathbf{k}4^{d-k}b^{k-d}\cdot c_{\mathbf{k}} \int_{B_\mathbf{k}}\left|f(a_tu_{\mathbf{s}}v_{b\mathbf{k}})-f(a_t n(\mathbf{z(\mathbf{s},\mathbf{r})}))\right|\,d\mathbf{r}\,d\mathbf{s}\\
    &\ll \mathcal S (f) b^{1/2}
\end{align*}
where on the last line we bounded the difference of functions by $\mathcal S(f)b^{1/2}$ and we used that $\sum_{\mathbf{k} }c_\mathbf{k}=1$.

Thus, it suffices to study the thickening $\int_{N}\varphi(n(\mathbf{z}))f(a_t n(\mathbf{z}))\,d\mathbf{z}$. Now, we introduce an extra average in single direction of the $U$-foliation. Consider the vector $(s,0,\ldots,0)$ for $s\in [0,\tau]$, where $\tau$ is a small parameter that we will specify later. By abuse of notation, we denote the vector $(s,0,\ldots,0)$ by $s$. Consider,
$$A = \frac{1}{\tau} \int_{[0,\tau]} \int_{N}\varphi(n(\mathbf{z}))f(a_tu_{s} n(\mathbf{z}))\,d\mathbf{z}\,ds.$$

We will eventually show that $\int_{N}\varphi(n(\mathbf{z}))f(a_t n(\mathbf{z}))\,d\mathbf{z}$ is comparable to $A$. To see this, notice that $m_U\left(u_{s}B_\mathbf{k}\triangle B_\mathbf{k}\right)\ll \tau m_U(B_\mathbf{k})$ since, $U\simeq\mathbb R^{k}$ and abelian groups have the Folner property. Hence,
\begin{align*}
    \left|\int_{N}\varphi(n(\mathbf{z}))f(a_tu_{s}n(\mathbf{z})x)\,d\mathbf{z}-\int_{N}\varphi(n(\mathbf{z}))f(a_tn(\mathbf{z})x)\,d\mathbf{z}\right|&\le \sum_{\mathbf{k}}c_{\mathbf{k}} \int_{u_{s}B_{\mathbf{k}}\triangle B_{\mathbf{k}}}|f(a_tn(\mathbf{z})x)|\,d\mathbf{z}\\
    &\ll \sum_{\mathbf{k}} c_{\mathbf{k}} \tau m_U(B_\mathbf{k})\|f\|_\infty\\
    &\ll \mathcal S(f) \tau.
\end{align*}
Integrating the above over $[0,\tau]$ and multiplying the above by $1/\tau$ yields
\begin{equation}\label{approx}
\left|\int_{N}\varphi(n(\mathbf{z}))f(a_tn(\mathbf{z})x)\,d\mathbf{z} - A\right|\ll \mathcal S(f)\tau\end{equation}

We choose $\tau$ to be of the form $e^{(\frac{1}{l}-2)t}$ for $l\ge 2.$ Then, by equation (\ref{approx}) and noting that $|\log(b)/4|\le t$ we have 
$$\left|\int_{N}\varphi(n(\mathbf{z}))f(a_t n(\mathbf{z}))\,d\mathbf{z} - A\right|\ll \mathcal S(f)b^{1/8},$$
and, as such, we have reduced our analysis to that of $A$.

By the Cauchy-Schwarz inequality, we have
$$|A|^2\le\int_{N} \left(\frac{1}{\tau}\int_{[0,\tau]} f(a_tu_{s} n(\mathbf{z}))\,ds\right)^2\varphi(n(\mathbf{z}))\,d\mathbf{z}.$$

Now that all the terms we are dealing with are non-negative, by utilizing the upper bound on $\varphi$ we deduce
\begin{align}\label{ineq: main}
|A|^2 &\le \frac{C4^{d-k}b^{-\varepsilon}}{\tau^2} \int_0 ^\tau \int_0 ^\tau\int_{[0,1]} \hat f_{s_1,s_2}(a_tn(\mathbf{z})x)\,d\mathbf{z}\,ds_1\,ds_2
\end{align}

where $\hat f_{s_1,s_2}(y)=f(a_t u_{s_1}a_{-t}y)f(a_tu_{s_2}a_{-t}y)$.

Observe that
$$\mathcal S(\hat f_{s_1,s_2})\le (e^{t}\tau)^\star\mathcal S(f)^2=e^{\frac{\star}{l}t}\mathcal S(f)^2.$$

By the above and by choosing a fixed $l'$  large enough so that $l =2\star l'\ref{K:effec}^{-1} \ge 2$, we have
$\mathcal S(\hat f_{s_1,s_2})\le e^{\ref{K:effec}t/2}\mathcal S(f)^2.$

Combining this observation and by our assumption on the effective equidistribution of the unstable foliation we deduce
$$
\int_{[0,1]^d} \hat f_{s_1,s_2}(a_tn(\mathbf{z})x)\,d\mathbf{z}\ll \int_X \hat f_{s_1,s_2}\,dm_X+e^{-\ref{K:effec}t/2}\mathcal S(f)^2\eta^{-L}.
$$

Combining the above with Equation (\ref{ineq: main}) yields
\begin{equation}\label{ineq: two terms}
    |A|^2\ll\frac{1}{\tau^2}\int_0 ^\tau \int_0 ^\tau \bigg(Cb ^{-\varepsilon}\int_X \hat f_{s_1,s_2}\,dm_X+Cb ^{-\varepsilon}e^{-\ref{K:effec}t/2}\mathcal S(f)^2\eta^{-L}\bigg)\,ds_1\,ds_2.
\end{equation}

Now we analyze the first term on the right of inequality (\ref{ineq: two terms}). We do this by splitting the integral over $[0,\tau]^2$ into two regions; one where $|s_1-s_2|>e^{-t}e^{t/2l}$ and so we can take advantage of the decay of correlations (Corollary \ref{Cor:Decay of Cor}) and one where $|s_1-s_2|<e^{-t}e^{t/2l}$ where we use that the region is of small measure. 

When $|s_1-s_2|>e^{-t}e^{t/2l}$, then by the decay of correlations Proposition \ref{Cor:Decay of Cor}  applied to the simple Lie group $G$, and observing that $d(e,u_{e^{t}(s_1-s_2)})\ge e^{t}|s_1-s_2|>e^{t/2l}>\frac{t}{2l}$ , we have
\begin{align*}
\int_X \hat f_{s_1,s_2}( x)\,dm_X &\ll e^{-\kappa_ X e^{2t}|s_1-s_2|} \mathcal S(f)^2\\
 &\le e^{-\kappa_ X \frac{t}{2l}} \mathcal S(f)^2.
\end{align*}
Thus, after integrating over $[0,\tau]^2$ and dividing by $\tau^2$, we have the following bound on the first term on the right of (\ref{ineq: two terms}) of
$$Cb ^{-\varepsilon}\int_X \hat f_{s_1,s_2}( x)\,dm_ X\ll Cb^{-\varepsilon}e^{-\kappa_X \frac{t}{2l}} \mathcal S(f)^2$$
whenever $|s_1-s_2|>e^{-t}e^{\frac{t}{2l}}.$ 

%%%%%%%%%%%%%%%%%%%%%

Finally, we consider the region close to the diagonal ($|s_1-s_2|<e^{-t}e^{t/2l}=\tau e^{-t/2l}$). Notice this region has area $2\tau^2e^{-t/2l}.$ By invariance of $m_ X$ and using that the $\|\cdot\|_\infty$ is dominated by $\mathcal S(\cdot)$ we obtain
$$\int_X\hat f_{s_1,s_2}( x) \,dm_ X =\int_X f(u_{e^{2t}(s_1-s_2)} x)f( x)\,dm_ X\le 2\tau^2e^{-t/2l}\|f\|_\infty ^2\le 2\tau^2e^{-t/2l}\mathcal S(f)^2.$$
Integrating over the region $|s_1-s_2|<e^{-t}e^{t/2l}$ and dividing everything by $\tau^2$ we get that the first term of the right side of inequality (\ref{ineq: two terms}) is smaller than
$$\ll \frac{1}{\tau^2}\mathcal S(f)^2\cdot 2\tau^2e^{-t/2l}\ll \mathcal S(f)^2\cdot e^{-t/2l}.$$
In total, we get this estimate of the right side of inequality (\ref{ineq: two terms}),
\begin{align*}
    |A|^2&\ll C\eta^{-L}\mathcal S(f)^2b^{-\varepsilon }\left(e^{-\kappa_Xt/2l}+e^{-t/2l}+ e^{-\ref{K:effec}t/2}\right)\\&\ll C
    \eta^{-L}\mathcal S(f)^2b^{-\varepsilon }e^{-\kappa_X t/2l}.
\end{align*}

We recall that we have $l=2\star \ell'\ref{K:effec}^{-1}$ and $e^{-4t}\le b$.  By choosing $$ \varepsilon\le \frac{\kappa_X \ref{K:effec}}{128\ell'}$$  we get 
$$|A|^2\ll  C\mathcal S(f)^2 \eta^{-L}e^{-t\kappa_X \ref{K:effec}/32l'}$$
and this finishes the proof after noting that $\ref{K:effec}=O(\kappa_X)$.
\end{proof}
\begin{comment}
Let $\eta=e^{-\varsigma t}$ for some $\varsigma\in(0,1)$
$\varsigma\le\frac{\kappa_X \ref{K:effec}}{8\ell'nd}$.
\end{comment}

\end{document}